\documentclass[11pt]{amsart}
\usepackage[leqno]{amsmath}
\usepackage{amssymb}
\usepackage{enumerate}
\usepackage[all]{xy}
\usepackage{graphicx}
\usepackage{tabularx}
\usepackage{caption}
\usepackage{subfigure}
\usepackage{comment}
\usepackage{xcolor}

\definecolor{ceylon}{RGB}{20,190,170}
\definecolor{darkblue}{RGB}{0,50,190}
\definecolor{darkred}{RGB}{190,0,0}

\usepackage{hyperref}
\hypersetup{
    colorlinks=true,
    linkcolor=darkblue,
    citecolor=darkred,
    urlcolor=blue,
    pdfborder={0 0 0}
}

\makeatletter
\newcommand{\labitem}[2]{
\def\@itemlabel{\textbf{#1}}
\item
\def\@currentlabel{#1}\label{#2}}
\makeatother

\newtheorem{theorem}{Theorem}[section]
\newtheorem{lemma}[theorem]{Lemma}
\newtheorem{proposition}[theorem]{Proposition}

\theoremstyle{remark}

\numberwithin{equation}{section}

\newcommand{\C}{\mathbf{C}}
\newcommand{\D}{\mathbf{D}}
\newcommand{\E}{\mathbf{E}}

\newcommand{\N}{\mathbf{N}}
\newcommand{\Z}{\mathbf{Z}}
\newcommand{\p}{\mathbf{P}}

\newcommand{\R}{\mathbf{R}}
\newcommand{\s}{\mathbf{S}}
\newcommand{\h}{\mathbf{H}}

\newcommand{\CE}{\mathcal {E}}
\newcommand{\CF}{\mathcal {F}}

\newcommand{\CT}{\mathcal {T}}

\newcommand{\CW}{\mathcal {W}}

\newcommand{\CG}{\mathcal {G}}

\newcommand{\dist}{{\mathrm{dist}}}

\newcommand{\im}{{\mathrm{Im}}}

\newcommand{\SLE}{{\mathrm{SLE}}}
\newcommand{\CLE}{{\mathrm{CLE}}}

\newcommand{\wt}{\widetilde}
\newcommand{\wh}{\widehat}

\newcommand{\eps}{\varepsilon}

\DeclareMathOperator{\diam}{diam}
\DeclareMathOperator{\area}{area}
\DeclareMathOperator{\dis}{dist}

\usepackage[normalem]{ulem}

\begin{document}

\title[Uniqueness of the welding problem for $\SLE$ and Liouville quantum gravity]{Uniqueness of the welding problem \\ for $\SLE$ and Liouville quantum gravity}
\author{Oliver McEnteggart,  Jason Miller, Wei Qian}

\date{\today}
\maketitle

\begin{abstract}
We give a simple set of geometric conditions on curves $\eta$, $\wt{\eta}$ in $\h$ from $0$ to $\infty$ so that if $\varphi \colon \h \to \h$ is a homeomorphism which is conformal off $\eta$ with $\varphi(\eta) = \wt{\eta}$ then $\varphi$ is a conformal automorphism of $\h$.  Our motivation comes from the fact that it is possible to apply our result to random conformal welding problems related to the Schramm-Loewner evolution ($\SLE$) and Liouville quantum gravity (LQG).  In particular, we show that if $\eta$ is a non-space-filling $\SLE_\kappa$ curve in $\h$ from $0$ to $\infty$, and $\varphi$ is a homeomorphism which is conformal on $\h \setminus \eta$, and $\varphi(\eta)$, $\eta$ are equal in distribution, then $\varphi$ is a conformal automorphism of $\h$.  Applying this result for $\kappa=4$ establishes that the welding operation for critical ($\gamma=2$) Liouville quantum gravity (LQG) is well-defined.  Applying it for $\kappa \in (4,8)$ gives a new proof that the welding of two independent $\kappa/4$-stable looptrees of quantum disks to produce an $\SLE_\kappa$ on top of an independent $4/\sqrt{\kappa}$-LQG surface is well-defined.
\end{abstract}

\section{Introduction}

\subsection{Overview}

Suppose that $\D_1$, $\D_2$ are copies of the unit disk $\D$ and $\phi$ is a homeomorphism from $\partial \D_1$ to $\partial \D_2$.  A \emph{conformal welding} of $\D_1$, $\D_2$ using the identification $\phi$ is a conformal structure on the sphere $\s^2$ obtained by identifying $\partial \D_1$ with $\partial \D_2$ according to $\phi$.  More precisely, it corresponds to a simple loop $\eta$ on $\s^2$ and two conformal transformations $\psi_1$, $\psi_2$ which take $\D_1$, $\D_2$ to the two components of $\s^2 \setminus \eta$ with $\phi = \psi_1 \circ \psi_2^{-1}$.  Given such a homeomorphism $\phi$, the two basic questions that one is led to ask are: (i) Does a conformal welding exist? (ii) If so, is it unique?  The main focus of the present article is on the latter question.

Recall that a set $K \subseteq \C$ is said to be \emph{conformally removable} if it has the property that whenever $U,V \subseteq \C$ are domains with $K \subseteq U$ and $\varphi \colon U \to V$ is a homeomorphism which is conformal on $U \setminus K$ then~$\varphi$ is conformal on all of $U$.  The uniqueness of a conformal welding is equivalent to the conformal removability of the interface~$\eta$.  There are several geometric conditions associated with a curve~$\eta$ which are known to imply that it is conformally removable.  For example, it was shown by Jones and Smirnov \cite{js2000remove} that boundaries of H\"older domains are conformally removable.  We recall that a simply connected domain $D \subseteq \C$ is a H\"older domain if there exists a conformal transformation $\varphi \colon \D \to D$ which is H\"older continuous up to $\partial \D$.  See also the works \cite{j1995remove,kw1996remove,kn2005remove,n2017remove} for other conditions which imply conformal removability.  In the present work, we will prove uniqueness results for conformal weldings in the setting in which the interface~$\eta$ is not the boundary of a H\"older domain or even a connected domain.  As we will explain in more detail below, our uniqueness results apply for conformal weldings under which the interface satisfies some regularity conditions and are therefore weaker than proving conformal removability.  Conformal removability questions in the setting in which the interface is not the boundary of a connected domain are subtle; it has long been known that the standard Sierpinski carpet is not conformally removable (see the introduction of \cite{n2017remove}) and it has been recently shown that the Sierpinski gasket is not conformally removable \cite{n2018gasketremove}.  The regularity conditions that we impose will allow us to circumvent some of the challenges associated with domains which are the complement of a carpet, but at the same time yield uniqueness results in the setting in which we are interested.

In recent years, there has been considerable interest in \emph{random} conformal weldings.  We will be focused on the case in which the welding interface $\eta$ is an instance of the Schramm-Loewner evolution ($\SLE$).  We recall that $\SLE$ is a random fractal curve defined in a simply-connected planar domain $D$.  It was introduced by Schramm  \cite{s2000scaling} in 1999 as a candidate to describe the scaling limits of lattice models in two-dimensional statistical mechanics.  $\SLE$'s have found many other applications in the intervening years, one of which is in the study of a certain theory of \emph{random} surfaces called Liouville quantum gravity (LQG).   In this context, $\SLE$'s arise as the gluing interface when one conformally welds two such surfaces with boundary \cite{s2016zipper,dms2014mating}.  (Let us also mention the work \cite{ajks2011welding} which considers the conformal welding of an  LQG (random) surface to a Euclidean (deterministic) disk; it turns out in this case that the resulting interface is not an $\SLE$.)  It is explained in \cite{s2016zipper} that one has uniqueness in this context when the gluing interface is an $\SLE_\kappa$ for $\kappa \in (0,4)$ and in \cite{dms2014mating} when $\kappa \in (4,8)$.  (Recall that $\SLE_\kappa$ curves are simple for $\kappa \leq 4$, self-intersecting but not space-filling for $\kappa \in (4,8)$, and space-filling for $\kappa \geq 8$ \cite{rs2005basic}.)  Prior to the present work, uniqueness had not been established for $\kappa=4$.  We will describe this in more detail and provide additional background below.  The purpose of the present work is to give a unified treatment of the uniqueness question for such conformal weldings which will be applicable for all $\kappa \in (0,8)$, and in particular $\kappa=4$.

\subsection{Main results}

The following theorem is one of the main results of this paper, which implies that there is at most one solution to any random conformal welding problem among the set of laws in which the gluing interface is a non-space-filling $\SLE$ curve.

\begin{theorem}
\label{thm:main_result}
Fix $\kappa \in (0,8)$.  Let $\eta$ be an $\SLE_\kappa$ curve in $\h$ from $0$ to $\infty$.  Suppose that $\varphi \colon \h \to \h$ is a homeomorphism which is conformal in $\h \setminus \eta$ and such that $\varphi(\eta) \stackrel{d}{=} \eta$.  Then $\varphi$ is a.s.\ a conformal automorphism of $\h$.
\end{theorem}

We note that studying the welding problem in the context of $\h$ is equivalent to studying the welding problem in $\s^2$ except one only welds part instead of all of the boundary.  In particular, we described conformal welding as the operation of gluing a pair of copies $\D_1, \D_2$ of the unit disk to produce $\s^2$ decorated by a path according to some homeomorphism $\phi$ from $\partial \D_1$ to $\partial \D_2$.  However, if we only weld a connected segment $I_1 \subseteq \partial \D_1$ (which is not all of $\partial \D_1$) with its image $I_2=\phi(I_1)$, then we can obtain a simply connected domain which we can take to be $\h$ and we can take the gluing interface to be a curve from $0$ to $\infty$.

Theorem~\ref{thm:main_result} also applies to a more general type of welding problem when $\kappa\in(4,8)$. In this range, $\SLE_\kappa$ a.s.\ intersects (without crossing) itself, and arises as the gluing interface of a countable number of disks (or of stable looptrees, see Section~\ref{sec:applications} for more details).

We remark that the part of Theorem~\ref{thm:main_result} for $\kappa \in (0,4)$ is not new.  The reason is that the work \cite{rs2005basic} implies that the complementary components of an $\SLE_\kappa$ curve for $\kappa \in (0,4)$ are a.s.\ H\"older domains and, as mentioned above, \cite{js2000remove} implies that boundaries of H\"older domains are conformally removable.  The range $\kappa \in [4,8)$ in Theorem~\ref{thm:main_result}, however, is new.  Indeed, it is not known whether $\SLE_4$ curves are conformally removable.  In particular, it is shown in \cite{gms2018multifractal} that an $\SLE_4$ curve a.s.\ does not form the boundary of a H\"older domain.  \cite[Corollary~4]{js2000remove} contains a weaker modulus of continuity condition than being the boundary of a H\"older domain which was further improved upon in \cite{kn2005remove} but it is not known whether $\SLE_4$ satisfies the sufficient conditions for conformal removability from \cite{js2000remove,kn2005remove}.  For $\kappa\in(4,8)$, since $\SLE_{\kappa}$ has double points, they have a carpet-like structure and conformal removability in this context is not well-understood (see \cite{n2017remove,n2018gasketremove}).

Theorem~\ref{thm:main_result} in fact follows from a more general result, where the condition of $\eta$ and $\varphi(\eta)$ being $\SLE$ curves can be weakened to a pair of deterministic geometric conditions.  Before describing these conditions, let us mention that the first condition is stable under the application of a locally bi-H\"older continuous homeomorphism $\h \to \h$, the second condition is stable under the application of a diffeomorphism $\h \to \h$, and we require that $\eta$ satisfies one of the conditions and $\varphi(\eta)$ satisfies the other one.  Since both of these conditions are satisfied by $\SLE_\kappa$ curves with $\kappa \in (0,8)$, we can formulate stronger versions of Theorem~\ref{thm:main_result}.  For example, Theorem~\ref{thm:main_result} remains true if we assume that $\varphi(\eta)$ is given by the image of an $\SLE_\kappa$ curve (for any value of $\kappa \in (0,8)$) under a locally bi-H\"older continuous homeomorphism $\h \to \h$.  We also do not have to assume \emph{a priori} that $\eta$, $\varphi(\eta)$ have the same $\kappa$ values.  There are also other versions of Theorem~\ref{thm:main_result} which hold under even weaker hypotheses.  As we will explain in more detail in Section~\ref{sec:applications}, the particular formulation given in Theorem~\ref{thm:main_result} is the one most relevant in the context of LQG.

We will now describe the conditions required for the general theorem statement.  Let $\eta$ be a  curve in~$\h$ from~$0$ to~$\infty$, i.e., $\eta \colon \R_+ \to \h$ is continuous with $\eta(0) = 0$ and $\lim_{t \to \infty} \eta(t) = \infty$. We also assume that $\eta$ is non-self-crossing, but allow it to be self-intersecting.  Let us first fix some notation. 
\begin{enumerate}
\item\label{1} For any $t>0$ and $\delta>0$, let $\tau$ (resp.\ $\sigma$) be the first (resp.\ last) time after (resp.\ before) $t$ that $\eta$ reaches $\partial B(\eta(t),\delta)$ and we denote by $\eta(t;\delta)$ the excursion $\eta[\sigma, \tau]$.
\item For any $z\in\h$ and $\eps \in (0,\delta)$, let us define the excursions of $\eta$ between  $\partial B(z, \eps)$ and $\partial B(z, \delta)$:  if there exists $t$ such that $\eta(t)\in B(z,\eps)$, then we let $\tau$ (resp.\ $\sigma$) be the first (resp.\ last) time after (resp.\ before) $t$ that $\eta$ reaches $\partial B(z,\delta)$ and we say that $\eta[\sigma, \tau]$ is an excursion between  $\partial B(z, \eps)$ and $\partial B(z, \delta)$.  The number of excursions of $\eta$ between $\partial B(z, \eps)$ and $\partial B(z, \delta)$ is always finite, because $\eta$ is a continuous curve with $\eta(t) \to \infty$ as $t \to \infty$.
\end{enumerate}
Let us now describe the following hypotheses on $\eta$:

\begin{enumerate}
\labitem{H1}{itm:H1} \textbf{Bounded number of crossings} (see Figure~\ref{cond1}): For any compact rectangle $K\subseteq \h$ and any $\beta \in (0,1)$, there exist $M>0$ and $\eps_0>0$, such that for all $\eps \in (0,\eps_0)$, and for all $z\in K$, the number of excursions  of $\eta$ between $\partial B(z, \eps^{\beta})$ and $\partial B(z,\eps)$ is at most $M$.

\labitem{H2}{itm:H2} \textbf{Non-self-tracing} (see Figure~\ref{cond2}): For any compact rectangle $K\subseteq \h$ and any $\alpha>\xi>1$, there exists $\delta_0>0$ such that for any $\delta \in (0,\delta_0)$, for any $t>0$ such that $\eta(t) \in K$, one can find  a point $y$ such that

\begin{itemize}
\item[(i)] $B(y, \delta^\alpha) \subseteq B\left(\eta(t),\delta \right)\setminus \eta$ and $ B(y, 2\delta^\alpha)\cap \eta\not=\emptyset.$

\item[(ii)]
Let $O$ be the connected component of $B(\eta(t),\delta)\setminus \eta$ that contains $y$. For any point $a$ in $\partial O \setminus \eta(t;\delta)$, any path contained in $O\cup\{a\}$ which connects $y$ to $a$ must exit the ball $B(y, \delta^\xi)$.

\end{itemize}
\end{enumerate}

\begin{figure}[h!]
\centering
\begin{minipage}[t]{.49\textwidth}
  \centering
  \includegraphics[scale=0.85]{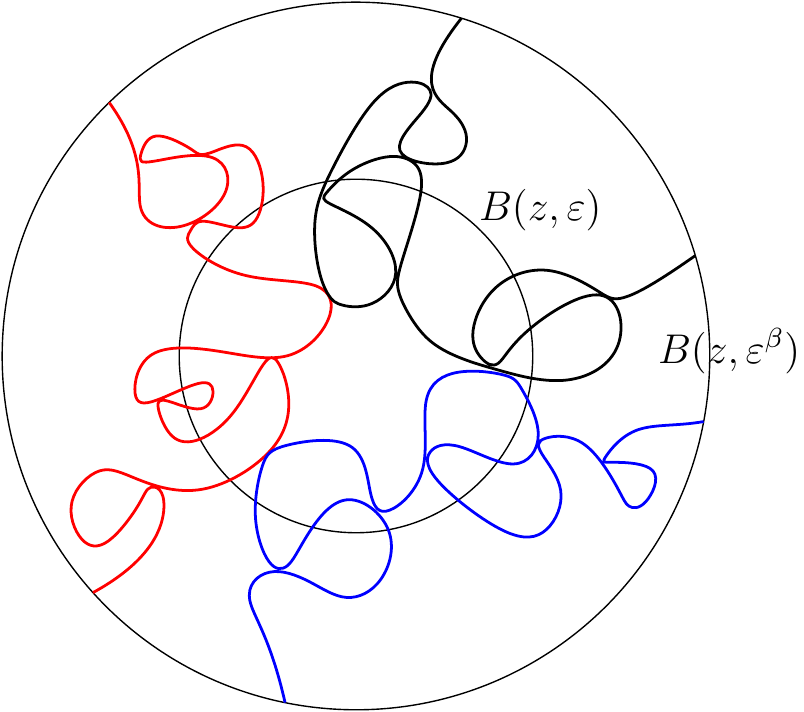}
  \captionof{figure}{\ref{itm:H1} bounded number of crossings across annuli. Here we depict three different crossings in three different colors. The curve can intersect itself, but never crosses itself.}
  \label{cond1}
\end{minipage}\,
\begin{minipage}[t]{.49\textwidth}
  \centering
  \includegraphics[scale=0.85]{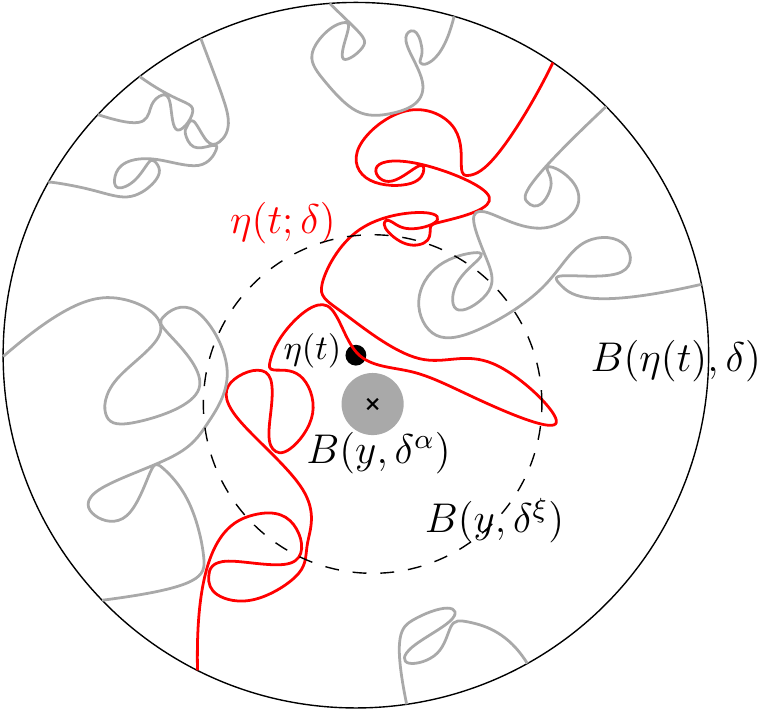}
  \captionof{figure}{\ref{itm:H2} existence of a ball near an excursion (depicted in red) but ``far away" from other parts of the curve (depicted in grey) in the sense that any path connecting $y$ to the grey parts of the curve must either intersect $\eta(t;\delta)$ or exit the ball $B(y, \delta^\xi)$.}
  \label{cond2}
\end{minipage}
\end{figure}

\begin{theorem}
\label{thm:general_result}
Let $\eta$ be a non-self-crossing curve in $\h$ from $0$ to $\infty$.  Suppose that $\varphi \colon \h \to \h$ is a homeomorphism which is conformal in $\h \setminus \eta$.  If $\eta$ satisfies~\ref{itm:H1} and has zero Lebesgue measure, $\varphi(\eta)$ satisfies~\ref{itm:H2} and has upper Minkowski dimension $d<2$, then $\varphi$ is a conformal automorphism of $\h$.
\end{theorem}

Let us emphasize that the conditions~\ref{itm:H1} and~\ref{itm:H2} involve no randomness, hence Theorem~\ref{thm:general_result} is a statement for deterministic curves.
In particular, the proof of Theorem~\ref{thm:general_result} does not involve SLE or LQG. 
We will then prove Theorem~\ref{thm:main_result} by checking that $\SLE_\kappa$ curves with $\kappa\in(0,8)$ a.s.\ satisfy the hypotheses~\ref{itm:H1} and~\ref{itm:H2}.  We note that $\SLE_\kappa$ for $\kappa\in(0,8)$ in fact satisfies much stronger geometric conditions than are assumed in~\ref{itm:H1} and~\ref{itm:H2}.  We believe that it is also possible to check the hypotheses~\ref{itm:H1} and~\ref{itm:H2} for any type of non-space-filling $\SLE$-type process, such as the exotic $\SLE_\kappa^\beta(\rho)$ processes considered in \cite{ms2017cleperc,ms2016gfflightcone,m2016lightconedim} or the conformal loop ensembles for $\kappa \in (8/3,8)$ \cite{she2009cle,sw2012cle}, but we will not carry this out here.

\subsection{Outline}
\label{subsec:outline}

The remainder of this article is structured as follows.  In Section~\ref{sec:applications}, we will describe the main application of Theorem~\ref{thm:main_result}, which is in the context of LQG.  In Section~\ref{sec:proof1.2}, we will prove Theorem~\ref{thm:general_result}. In Section~\ref{sec:non_tracing}, we will show that $\SLE_\kappa$ curves with $\kappa\in(0,8)$ a.s.\ satisfy the hypotheses~\ref{itm:H1} and~\ref{itm:H2}, hence proving Theorem~\ref{thm:main_result}.  We emphasize that the proof of Theorems~\ref{thm:main_result} and~\ref{thm:general_result} will not use LQG.  In particular, it is not necessary to understand Section~\ref{sec:applications} in order to understand the proofs of the main results.

\subsection*{Acknowledgements}

We thank an anonymous referee for many helpful comments which have improved the exposition throughout the article.  JM was supported by ERC Starting Grant 804166 (SPRS). WQ acknowledges the support of an Early Postdoc Mobility grant of the SNF, EPSRC grant EP/L018896/1, and a JRF of Churchill college.

\section{Applications to Liouville quantum gravity}
\label{sec:applications}

We will now provide some additional motivation and consequences of Theorem~\ref{thm:main_result}.  The contents of this section are not needed for the proof of Theorems~\ref{thm:main_result} and~\ref{thm:general_result}.

\subsection{Liouville quantum gravity review}
\label{subsec:lqg_review}

Suppose that $D \subseteq \C$ is a planar domain, $h$ is an instance of (some form of) the Gaussian free field (GFF) on $D$, and $\gamma \in (0,2]$ is a fixed parameter.  The Liouville quantum gravity (LQG) surface parameterized by $D$ and described by $h$ formally corresponds to the metric tensor
\begin{equation}
\label{eqn:lqg_metric}
e^{\gamma h(z)} (dx^2 + dy^2)
\end{equation}
where $dx^2 + dy^2$ represents the Euclidean metric on $D$.  The expression~\eqref{eqn:lqg_metric} does not make literal sense since $h$ is a distribution and does not take values at points.

In the case that $\gamma \in (0,2)$, the volume form associated with~\eqref{eqn:lqg_metric} was constructed in \cite{ds2011kpz}.  The approach taken in \cite{ds2011kpz} involves a certain regularization procedure.  Namely, for each $\eps > 0$ and $z \in D$ such that $B(z,\eps) \subseteq D$ we let $h_\eps(z)$ denote the average of $h$ on $\partial B(z,\eps)$.  One then takes
\begin{equation}
\label{eqn:mu_h_def}
 \mu_h^\gamma = \lim_{\eps \to 0} \eps^{\gamma^2/2} e^{\gamma h_\eps(z)} dz
\end{equation}
where $dz$ denotes Lebesgue measure on $D$.  The normalization factor $\eps^{\gamma^2/2}$ is necessary to obtain a non-trivial limit.  It is also possible to construct a measure in the critical case $\gamma=2$ \cite{drsv2014critical_deriv,drsv2014critical_kpz}.  In order to get a non-trivial limit, one has to introduce an extra correction in the normalization.  Following \cite{js2017uniqueness,hrv2018disk,p2018uniqueness}, one takes
\begin{equation}
\label{eqn:critical_mu_h_def}
\mu_h^{\gamma=2} = \lim_{\eps \to 0} \eps^2 \sqrt{\log \eps^{-1}} e^{2 h_\eps(z)} dz,
\end{equation}
where $dz$ again denotes Lebesgue measure on $D$.  (The works \cite{drsv2014critical_deriv,drsv2014critical_kpz} construct the critical LQG measure using a different approximation scheme.)

The regularization procedures in~\eqref{eqn:mu_h_def}, \eqref{eqn:critical_mu_h_def} lead to a certain change of coordinates formula for the measure $\mu_h^\gamma$.  Namely, suppose that $\varphi \colon \wt{D} \to D$ is a conformal transformation and
\begin{equation}
\label{eqn:map_equiv}
 \wt{h} = h \circ \varphi + Q \log| \varphi'| \quad\text{where}\quad Q = \frac{2}{\gamma} + \frac{\gamma}{2},
 \end{equation}
then it is a.s.\ the case that for all Borel sets $A$ one has that $\mu_h^\gamma (\varphi(A)) = \mu_{\wt{h}}^\gamma(A)$.

We say that two domain/field pairs $(D,h)$, $(\wt{D},\wt{h})$ are equivalent as quantum surfaces if $h$, $\wt{h}$ are related as in~\eqref{eqn:map_equiv}.  A \emph{quantum surface} is an equivalence class under this equivalence relation.  A choice of representative of a quantum surface is referred to as an \emph{embedding} of the quantum surface.  One can similarly extend these definitions to the setting of surfaces with extra marked points or a distinguished path.

We remark that whether two embeddings describe an equivalent quantum surface can in some cases be a subtle question.  For example, two definitions of LQG on the sphere are respectively given in \cite{dms2014mating} and \cite{dkrv2016sphere} which on the surface appear to be very different.  It was later proved in \cite{ahs2017equiv} that the constructions of \cite{dms2014mating,dkrv2016sphere} give rise to equivalent quantum surfaces.

In the case that $h$ has free boundary conditions on a linear boundary segment $L \subseteq \partial D$, one can similarly define a boundary length measure $\nu_h^\gamma$ by setting
\begin{align}
\nu_h^\gamma &= \lim_{\eps \to 0} \eps^{\gamma^2/4} e^{\gamma h_\eps(z)/2} dz \quad\text{for}\quad \gamma \in (0,2)\\
\nu_h^{\gamma=2} &= \lim_{\eps \to 0} \eps \sqrt{\log \eps^{-1}} e^{h_\eps(z)} dz \quad\text{for}\quad \gamma=2,
\end{align}
where in each case $dz$ denotes Lebesgue measure on $L$.  In the case that $h$ has free boundary conditions on part of $\partial D$ which is not a linear segment, one can conformally map $D$ to a domain which has piecewise linear boundary, define the boundary measure as above, and then map back using~\eqref{eqn:map_equiv}.

We remark that a general theory of random measures which have the same law as $\mu_h^\gamma$ and $\nu_h^\gamma$ was developed earlier by Kahane and is referred to as \emph{Gaussian multiplicative chaos} \cite{k1985gmc}.  See also \cite{rv2014gmc} for a more recent review.  See also \cite{rv2010revisited,b2017gmc}.  Similar measures also appeared earlier in \cite{hk1971quantum}.

We also remark that the metric (i.e., distance function) for LQG was first constructed in the case that $\gamma = \sqrt{8/3}$ in \cite{ms2015lqg_tbm1,ms2016lqg_tbm2,ms2016lqg_tbm3} and recently for all $\gamma \in (0,2)$ in \cite{dddf2019tightness,gm2019local,dfgps2019weak,gm2019conf,gm2019metric,gm2019coord}.

The study of LQG surfaces is motivated in part because they have been conjectured to describe the scaling limits of random planar maps decorated by an instance of a statistical physics model.  There are a number of different ways of formulating such a conjecture depending on the topology that one chooses.  Scaling limit results of this type have now been proved in a number of cases using the so-called peanosphere topology \cite{lsw2017schnyder,kmsw2015bpmap,she2016inventory,gkmw2018bending} and in the Gromov-Hausdorff topology \cite{gm2016saw,gm2017perc}.  For example, the cases $\gamma=1,\sqrt{4/3}, \sqrt{2},\sqrt{8/3},\sqrt{3}$ respectively correspond to random planar maps decorated by a Schnyder woods, bipolar orientation, uniform spanning tree, percolation configuration, Ising (or FK-Ising) model.  The case $\gamma=2$, which is one of the main motivations for the present article, conjecturally corresponds to a random planar map decorated by an instance of the $4$-state Potts model.

\subsection{Welding quantum surfaces}
\label{subsec:welding}

\begin{figure}[ht!]
\begin{center}
\includegraphics[scale=0.85]{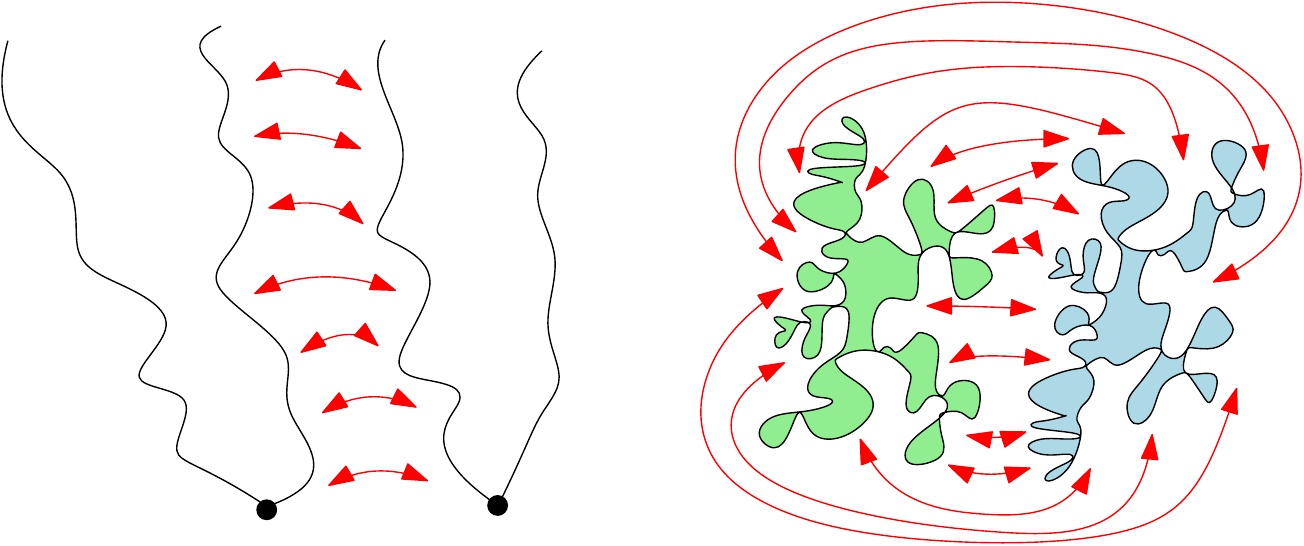}
\end{center}
\caption{\label{fig:welding} {\bf Left:} It is shown in \cite{s2016zipper} that it is possible to conformally weld two independent LQG surfaces called quantum wedges along their boundary rays to produce an LQG surface decorated by a simple $\SLE_\kappa$ curve.  {\bf Right:} It is shown in \cite{dms2014mating} that if one glues together two independent $\kappa'/4$-stable looptrees of quantum disks, $\kappa'=4/\gamma^2 \in (4,8)$, then one obtains an LQG surface decorated by a self-interesting but non-space-filling $\SLE_{\kappa'}$ curve.}
\end{figure}

A number of different welding operations for quantum surfaces are considered in \cite{s2016zipper,dms2014mating}.  Fix $\gamma \in (0,2)$ and let $\kappa = \gamma^2$. In this case, $\SLE_\kappa$ is a.s.\ a simple curve. The basic idea is that if one takes an appropriate type of $\gamma$-LQG surface $\CW = (\h,h,0,\infty)$ parameterized by $\h$ and with marked points at $0$ and $\infty$ and then draws an independent $\SLE_\kappa$ process $\eta$ in $\h$ from $0$ to $\infty$ on top of it, then the quantum surfaces $\CW_1$, $\CW_2$ parameterized by the components of $\h \setminus \eta$ which are to the left and right of $\eta$ and marked by $0$ and $\infty$ are independent.  Moreover, the path-decorated quantum surface $(\CW,\eta)$ can be recovered as a conformal welding of $\CW_1$, $\CW_2$, where the boundary welding homeomorphism is provided by the $\gamma$-LQG boundary measure.  The uniqueness of the welding follows from the conformal removability of $\SLE_\kappa$ for $\kappa \in (0,4)$.  Indeed, suppose that $(\wt{\CW},\wt{\eta})$ is another path-decorated quantum surface such that the quantum surfaces $\wt{\CW}_1$, $\wt{\CW}_2$ parameterized by the components of $\h \setminus \wt{\eta}$ which are to the left and right of $\wt{\eta}$ are equivalent as quantum surfaces to $\CW_1$, $\CW_2$, respectively.  This means that there exist conformal maps $\varphi_j$, $j=1,2$, from $\CW_j$ to $\wt{\CW}_j$ so that if $h_j$ (resp.\ $\wt{h}_j$) is the field which describes $\CW_j$ (resp.\ $\wt{\CW}_j$) then we have that $h_j = \wt{h}_j \circ \varphi_j + Q\log|\varphi_j'|$.  If $\wt{\CW}_1$, $\wt{\CW}_2$ are identified according to $\gamma$-LQG boundary length, then $\varphi_1,\varphi_2$ extend to a homeomorphism $\varphi \colon \CW \to \wt{\CW}$ which is conformal off $\eta$.  The conformal removability of $\SLE_\kappa$ for $\kappa \in (0,4)$ implies that $\varphi$ is conformal everywhere, hence $(\CW,\eta)$, $(\wt{\CW},\wt{\eta})$ are equivalent as path-decorated quantum surfaces.  Extensions of the same idea also apply when one considers quantum surfaces with other topologies (e.g., $\C$ and decorated by an independent whole-plane $\SLE_\kappa$).

The existence of the welding in the critical case $\gamma=2$ was recently proved in \cite{hp2018welding}.  The uniqueness of the welding in the case $\gamma=2$ follows from Theorem~\ref{thm:main_result}.  Combined, this implies that the welding operation for critical $(\gamma=2)$ LQG is well-defined.  To explain this in more detail, suppose that $\CW = (\h,h,0,\infty)$ is a quantum surface parameterized by $\h$ with marked points at $0$ and $\infty$ and that $\eta$ is an independent $\SLE_4$ on $\h$ from $0$ to $\infty$.  Let $\CW_1$, $\CW_2$ be the quantum surfaces parameterized by the components of $\h \setminus \eta$ which are to the left and right of $\eta$.  Suppose that $(\wt{\CW},\wt{\eta})$ is another path-decorated quantum surface which has the same law as $(\CW,\eta)$ so that the quantum surfaces $\wt{\CW}_1,\wt{\CW}_2$ parameterized by the components of $\h \setminus \wt{\eta}$ which are to the left and right of $\wt{\eta}$ are equivalent to $\CW_1,\CW_2$ and that $\wt{\CW}_1,\wt{\CW}_2$ are identified according to LQG boundary length.  Then there exists a homeomorphism $\varphi \colon \h \to \h$ which is conformal on $\h \setminus \eta$ which takes $\CW_j$ to $\wt{\CW}_j$ for $j=1,2$.  In particular, $\varphi(\eta) = \wt{\eta}$ has the same law as $\eta$.  Theorem~\ref{thm:main_result} implies that $\varphi$ is conformal everywhere so that $(\CW,\eta)$ and $(\wt{\CW},\wt{\eta})$ are equivalent as path-decorated quantum surfaces.  Consequently, $(\CW,\eta)$ is a.s.\ determined by $\CW_1,\CW_2$ (as the argument we have just described above implies that conditionally independent samples from the law of $(\CW,\eta)$ given $\CW_1,\CW_2$ must be a.s.\ the same).  This argument shows that there can be a.s.\ at most one conformal welding of $\CW_1$, $\CW_2$ in which the welding interface is an $\SLE_4$ type curve (or more generally any curve which satisfies the hypotheses of Theorem~\ref{thm:general_result}).  However, it does not rule out the existence of conformal weldings in which the welding interface exhibits much wilder behavior (i.e., the possibility that $\wt{\eta}$ does not satisfy the hypotheses of Theorem~\ref{thm:general_result}).  Indeed, this would require us to establish the removability of $\SLE_4$.

In \cite{s2016zipper,dms2014mating} it is shown that it is also possible to consider $\kappa' = 16/\gamma^2 > 4$ processes on top of LQG surfaces.  Suppose that we are in the setting that $\gamma \in (\sqrt{2},2)$ so that $\kappa' \in (4,8)$.  As such an $\SLE_{\kappa'}$ process has double points and separates non-trivial regions from its target point, the quantum surfaces which are cut out on the left and right sides of the path are not simply connected but rather have a tree-like structure.  The reader familiar with the results of \cite{dms2014mating} will recall that the entire path-decorated quantum surface can be mathematically described as a welding of independent $\kappa'/4$-stable looptrees of quantum disks.  Since it is not known if such $\SLE_{\kappa'}$ processes are conformally removable, a different type of argument for showing the conformal welding is well-defined is given in \cite{dms2014mating}.  The statement given in \cite{dms2014mating} is an abstract measurability result which says that the overall path-decorated quantum surface $(\CW,\eta')$, where $\eta'$ is an independent $\SLE_{\kappa'}$ process and $\CW$ is an appropriate type of quantum surface, is a.s.\ determined by the $\kappa'/4$-stable looptrees of quantum disks $\CT_1,\CT_2$ which are parameterized by the components cut off by $\eta'$ on its left and right sides.  Arguing as in the previous two paragraphs, Theorem~\ref{thm:main_result} gives another proof of this fact.  Moreover, it implies that this measurable function satisfies some properties which are not obvious from the proof given in \cite{dms2014mating}.  For example, it is not obvious that the abstract measurable function constructed in \cite{dms2014mating} behaves well under the operation of time-reversal.  More precisely, suppose that one has two independent $\kappa'/4$-stable loop trees $\CT_1$, $\CT_2$ of quantum disks,  then the welding of $\CT_1$, $\CT_2$ together determine a path decorated surface $(\CW,\eta')$.  One can also reverse the orientations of $\CT_1,\CT_2$ to obtain another pair $\wh{\CT}_1$, $\wh{\CT}_2$ of $\kappa'/4$-stable looptrees (since the law of a stable looptree is preserved when switching the orientation).  Then the welding of $\wh{\CT}_1$, $\wh{\CT}_2$ also determines a path decorated surface $(\wh{\CW},\wh{\eta}')$.  It follows from \cite{dms2014mating} that $(\wh{\CW},\wh{\eta}')$ has the same law as $(\CW,\eta')$ but it does not follow directly from \cite{dms2014mating} that $(\wh{\CW},\wh{\eta}_R')$, where $\wh{\eta}_R'$ is the time-reversal of $\wh{\eta}'$, is equal to $(\CW,\eta')$ as a path-decorated quantum surface  (one would need to use the reversibility of $\SLE_{\kappa'}$ for $\kappa' \in (4,8)$ proved in \cite{ms2016imag3} to obtain these statements).  However, since the geometric hypotheses of Theorem~\ref{thm:general_result} are satisfied by the time-reversal of any curve that satisfies them in the forward direction, the uniqueness statement obtained from the present article also behaves well with respect to time-reversal.  For example, Theorem~\ref{thm:main_result} holds if we assume that $\eta$ is an $\SLE_\kappa$ curve and $\varphi(\eta)$ has the law of the time-reversal of an $\SLE_\kappa$ curve.  These are now known to be the same, but we expect that one could use this argument to give a new proof of the reversibility of $\SLE_{\kappa'}$ for $\kappa' \in (4,8)$ (although we do not carry this out here since our proof that $\SLE$ satisfies the hypotheses of Theorem~\ref{thm:general_result} uses the reversibility of $\SLE$ for simplicity).

There is forthcoming work of the second author together with Sheffield and Werner \cite{msw2018cleonlqg} which will study conformal loop ensembles ($\CLE$) on Liouville quantum gravity.  We expect that the results established here will also lead to uniqueness results for weldings considered in that context.

We remark that the uniqueness results for the welding fall into the wider class of results which are concerned with showing that a certain object coupled with the GFF is in fact a.s.\ determined by the GFF.  Other important examples include:
\begin{itemize}
\item The level lines \cite{ss2013continuum_contour} and flow lines of the GFF \cite{dub2009partition,ms2016imag1,ms2016imag4}.
\item The matings of correlated continuum random trees to produce space-filling $\SLE_\kappa$ for $\kappa > 4$ on an LQG surface from \cite{dms2014mating}.
\item The metric measure space structure of the Brownian map determines its embedding as the $\sqrt{8/3}$-LQG sphere \cite{ms2015lqg_tbm1,ms2016lqg_tbm2,ms2016lqg_tbm3}.
\end{itemize}

\section{Proof of Theorem~\ref{thm:general_result}} \label{sec:proof1.2}

In this section, we assume that $\eta$ and $\wt\eta$ are non-self-crossing curves in $\h$ from $0$ to $\infty$. We assume that $\eta$ satisfies~\ref{itm:H1} and has zero Lebesgue measure and that $\wt\eta$ satisfies~\ref{itm:H2} and has upper Minkowski dimension $d<2$. Let $\varphi$ be a homeomorphism from $\h$ onto itself that is conformal on $\h\setminus \eta$ and such that $\wt\eta=\varphi(\eta)$. We want to show that $\varphi$ is conformal everywhere.

\subsection{Outline of the proof}

We know by the hypotheses that $\varphi$ is a homeomorphism which is a.e.\ conformal.  
In order to show that $\varphi$ is conformal everywhere, it suffices to show that it in addition has the ACL (absolutely continuous on lines) property (see \cite[Chapter II]{ahflors2006qc}), namely $\varphi$ is absolutely continuous on a.e.\ line which is parallel to one of the coordinate axes (i.e., the $x$-axis or the $y$-axis). 

To show that $\varphi$ is absolutely continuous on a given line $L$, we need to show that for each compact interval $I$ of $L$ and every $\eps > 0$ there exists $\delta > 0$ so that if $x_1,y_1,\ldots,x_k,y_k$ are points in $I$ with $\sum_{j=1}^k |y_j-x_j| < \delta$ then $\sum_{j=1}^k |\varphi(x_j) - \varphi(y_j)| < \eps$.  To prove that this is the case, we will rely on Lemma~\ref{lem:jensen} and Proposition~\ref{prop:distortion}.

\begin{lemma}\label{lem:jensen}
For any compact set $K\subseteq\h$, the function $\varphi'$ is $L^1$ on $K$.
\end{lemma}
\begin{proof}
Note that $\varphi'$ is only well-defined away from $\eta$, but since $\eta$ has zero Lebesgue measure, the integral of $|\varphi'|$ on $K$ is well-defined.  By the Cauchy-Schwarz inequality, we have
\begin{align*}
\int_K|\varphi'(w)| dw \le \left(\int_K |\varphi'(w)|^2 dw\right)^{1/2} \area(K)^{1/2} =\area (\varphi(K))^{1/2} \area(K)^{1/2} <\infty,
\end{align*}
where the equality is due to the area transformation formula for the conformal map $\varphi$ and we have used that $\area(\varphi(K)) = \area(\varphi(K \setminus \eta))$ as $\wt{\eta}$ has upper Minkowski dimension $d < 2$.
\end{proof}

Note that Lemma~\ref{lem:jensen} allows us to control the variation of $\varphi$ away from $\eta$. We will need the following proposition to control the variation of $\varphi$ across the curve $\eta$.

\begin{proposition}
\label{prop:distortion}
Suppose that $K\subseteq\h$ is a compact rectangle and let $z \in K$ be chosen uniformly at random.  Then for any $\iota>0$, we have
\begin{align}\label{eq:expected_distortion}
\E\!\left[\diam\!\left(\varphi(B(z,\eps))\right) \mathbf{1}_{d(z,\eta)<\eps}  \right] =O( \eps^{2/d-\iota}).
\end{align}
\end{proposition}

We emphasize that the expectation in Proposition~\ref{prop:distortion} is over the randomness in $z$.  We will prove Proposition~\ref{prop:distortion} in the later subsections. Let us first prove Theorem ~\ref{thm:general_result} assuming Proposition~\ref{prop:distortion}.

\begin{proof}[Proof of Theorem ~\ref{thm:general_result}]
As we have explained earlier, it is enough to prove that  $\varphi$ is absolutely continuous on a.e.\ line which is parallel to one of the coordinate axes. We will show this for horizontal lines, since it works the same way for vertical lines.

Fix $a_2>a_1$, $b_2>b_1>0$, and let $K$ be the compact rectangle $[a_1, a_2]\times [b_1, b_2]$.  We randomly choose $b\in [b_1, b_2]$ according to the uniform measure on $[b_1, b_2]$. Let $L$ be the random horizontal line at height~$b$. It suffices to prove that $\varphi$ is a.s.\ absolutely continuous on $L$ and to this end it is enough to control the behavior of $\varphi$ on the compact interval $I:=L\cap K$, since we can take any $a_2>a_1$.

Fix $\delta > 0$ and let $x_1, y_1, \ldots, x_k, y_k$ be points in $I$ such that $\sum_{j=1}^k |y_j-x_j| < \delta$.
We aim to bound the quantity $\Delta:= \sum_{j=1}^k |\varphi(x_j) - \varphi(y_j)| $.
For any $n \in \N$, we divide $I$ into $n$ intervals $I_1, \ldots, I_n$ of length $\alpha_n:=(a_2-a_1)/n$.
Let 
\[ S_n:=\sum_{j=1}^n \diam(\varphi(I_j))  \mathbf{1}_{\eta \cap I_j\not=\emptyset}.\]
Then
\begin{align}\label{eqq1}
\Delta\le S_n+   \sum_{j=1}^k \int_{ [x_j, y_j]} |\varphi'(w)| dw \quad\text{for all}\quad n \in \N.
\end{align}

If we choose $a$ uniformly in $[a_1, a_2]$, then the point $z=(a,b)$ is a uniformly random point in $K$.
For any $n \in \N$, we divide $K$ into $n\times n$ rectangles of size $\alpha_n \times \beta_n$ where $\alpha_n=(a_2-a_1)/n$ and $\beta_n= (b_2-b_1)/n$.
For $1 \le u,v \le n$, we denote by $R_{u,v}$ the rectangle with corners $(a_1+(u-1) \alpha_n, b_1+(v-1) \beta_n)$ and $(a_1+ u \alpha_n, b_1+ v \beta_n)$.

Letting $\E$ denote the expectation w.r.t.\ the random point $z = (a,b)$, we have that
\begin{align}
\E\!\left[\diam\!\left(\varphi(B(z,\alpha_n+\beta_n))\right) \mathbf{1}_{d(z,\eta)<\alpha_n+\beta_n}  \right] 
\ge& \sum_{u=1}^n  \sum_{v=1}^n  \E\!\left[ \diam(\varphi(R_{u,v}))  \mathbf{1}_{\eta\cap R_{u,v} \not=\emptyset} \mathbf{1}_{z\in R_{u,v}}  \right] \notag\\
=& \frac{1}{n^2} \sum_{u=1}^n  \sum_{v=1}^n   \diam(\varphi(R_{u,v}))  \mathbf{1}_{\eta\cap R_{u,v} \not=\emptyset}. \label{eqn:diam_lbd}
\end{align}
On the other hand, we know that
\begin{align}
\E[S_n] \le& \sum_{u=1}^n  \sum_{v=1}^n  \E\!\left[ \diam(\varphi(R_{u,v}))  \mathbf{1}_{\eta \cap R_{u,v} \not=\emptyset} \mathbf{1}_{(v-1)\beta_n\le b-b_1< v \beta_n} \right] \notag\\
=&\frac{1}{n} \sum_{u=1}^n\sum_{v=1}^n   \diam(\varphi(R_{u,v}))  \mathbf{1}_{\eta \cap R_{u,v} \not=\emptyset}. \label{eqn:sn_ubd}
\end{align}
Combining~\eqref{eqn:diam_lbd} and~\eqref{eqn:sn_ubd} and applying Proposition~\ref{prop:distortion} in the second to last equality, we see that
\begin{align*}
\E[S_n] \le n \E\!\left[\diam\!\left(\varphi(B(z,\alpha_n+\beta_n))\right) \mathbf{1}_{d(z,\eta)<\alpha_n+\beta_n}  \right]= n \times O(n^{\iota-2/d}) = o(1) \quad\text{as}\quad n \to \infty.
\end{align*}
This implies that $S_n$ converges to $0$ in probability, hence we can find a subsequence $n(r)$ along which $S_{n(r)}$ converges to $0$ a.s.

Putting the sequence $S_{n(r)}$ into~\eqref{eqq1} and letting $r$ go to $\infty$, we get that a.s.
\begin{align}\label{eqq2}
\Delta \le \sum_{j=1}^k \int_{ [x_j, y_j]} |\varphi'(w)| dw.  
\end{align}
We know by Lemma~\ref{lem:jensen} that $\varphi'$ is $L^1$ on $K$, hence $\varphi'$ is a.s.\ $L^1$ on $I$ (as the height of $I$ is uniformly random).  This implies that for any $\eps>0$, we can find $\delta_0>0$, such that for all $\delta \in(0,\delta_0)$ and all points $x_1,y_1,\ldots, x_k, y_k$ in $I$ such that $\sum_{j=1}^k |y_j-x_j| < \delta$, the right hand-side of~\eqref{eqq2} is smaller than $\eps$. This proves that it is a.s.\ the case that for such a randomly chosen line $L$, the function $\varphi$ is absolutely continuous on $L$.
\end{proof}

Our main goal in the rest of the section will be to prove Proposition~\ref{prop:distortion}.
This will be accomplished in two steps in Section~\ref{sec:roundish} and Section~\ref{sec:proof_of_proposition}. 
We will first estimate in Section~\ref{sec:roundish} the distortion under $\varphi$ of a small ball $B(z,\eps)$ which intersects $\eta$. Then in Section~\ref{sec:proof_of_proposition}, we will finally prove Proposition~\ref{prop:distortion} using the results in Section~\ref{sec:roundish} and the fact that the upper Minkowski dimension of $\wt\eta$ is strictly less than  $2$.

\subsection{Distortion along the curve}
\label{sec:roundish}

Throughout, we fix $\alpha>1> \beta' > \beta>\rho >0$ and a compact rectangle $K\subseteq \h$.  The goal of this section is to prove that provided $\eps > 0$ is sufficiently small, if a ball $B(z,\eps)$ with $z \in K$ intersects~$\eta$, then $\varphi(B(z,\eps^\rho))$ (as long as the image  is also small) contains a Euclidean ball with diameter at least $\diam(\varphi(B(z,\eps)))^\alpha$.  

Let us first recall the Beurling estimate (see, e.g., \cite[Theorem V.4.1]{MR1329542}), which is a basic tool that we will use multiple times in the sequel.
\begin{lemma}[Beurling estimate]
There exists a constant $c>0$ such that for any curve $\gamma$ from $\partial B(0,\eps)$ to the unit circle, the probability that a Brownian motion starting at $-\eps$ reaches the unit circle without hitting $\gamma$ is bounded above by $c \eps^{1/2}$.
By inversion symmetry, the probability that a Brownian motion starting at $1$ reaches $\partial B(0, \eps)$ without hitting $\gamma$ is also bounded above by $c \eps^{1/2}$.
\end{lemma}
Let us now come back to the estimates of the distortion along the curve.
Let $\eps > 0$.  For any $z\in K$ such that $\eta$ intersects $B(z,\eps)$, let $N$ be the number of excursions of $\eta$ between $\partial B(z, \eps^{\beta'})$ and $\partial B(z, \eps^\beta)$. By~\ref{itm:H1}, we know that there exist $M>0$ and $\eps_0>0$ such that for all $\eps\in(0,\eps_0)$ we have $N\le M$.  Denote the excursions by $e_1,\ldots, e_N$.  
See Figure~\ref{distortion} for an illustration of the definitions. 
Let $\delta_i$ be the diameter of $\varphi(e_i)$. Let $\wh\delta:=\max(\delta_1,\ldots, \delta_N)$.

\begin{lemma}\label{lem:roundish}
There exist $\eps_0, \delta_0\in(0,1)$ such that for any $z\in K$ and $\eps \in (0,\eps_0)$ with $\eta \cap B(z,\eps) \neq \emptyset$ there exists $y$ such that
\begin{align}\label{byd}
 B(y, 2\delta^\alpha)\cap \wt{\eta} \neq \emptyset \quad \text{and}\quad B(y, \delta^\alpha) \subseteq \varphi(B(z, \eps^\rho)) \quad\text{where}\quad \delta:=\min( \wh\delta, \delta_0).
\end{align}
\end{lemma}

\begin{figure}[h!]
\centering
\includegraphics[scale=0.85]{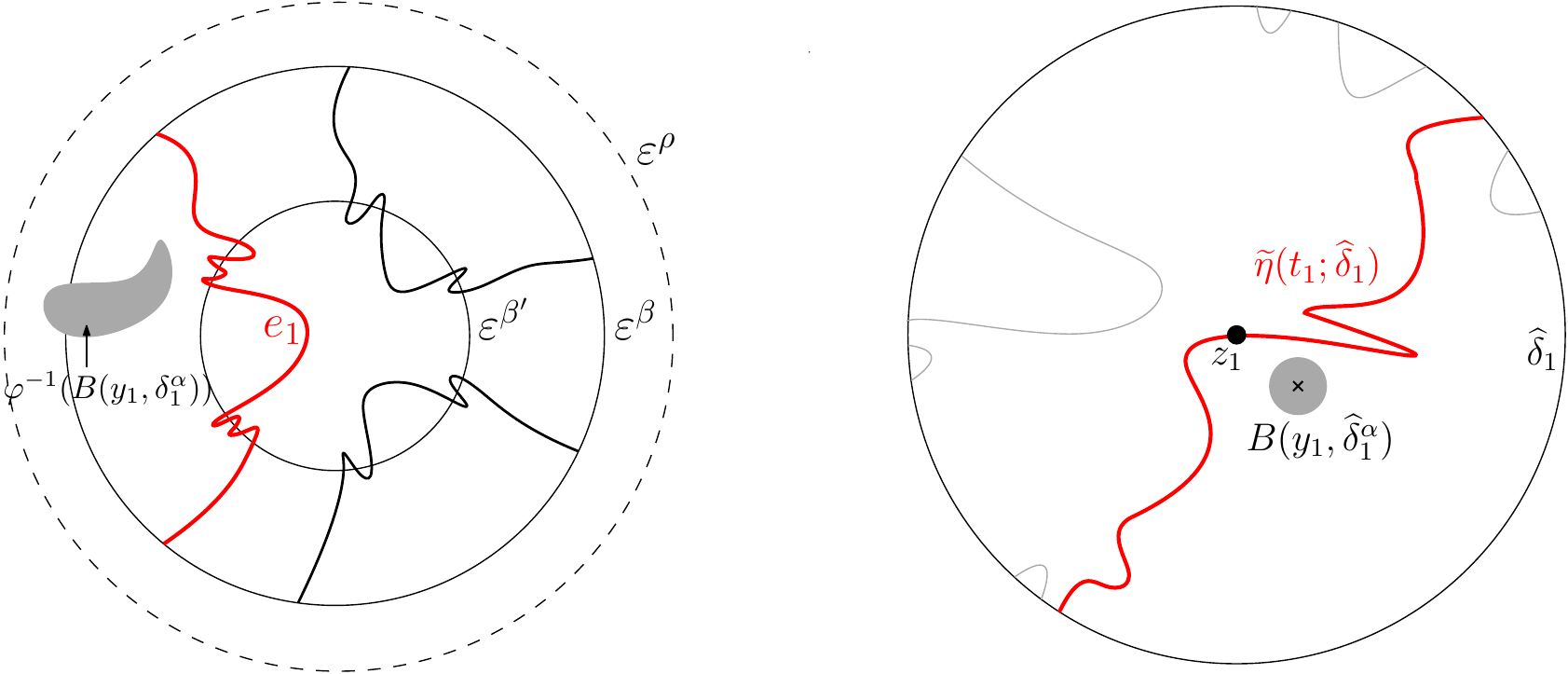}
\caption{On the left, $\eta$ makes three excursions $e_1, e_2, e_3$ between $\partial B(z,\eps^{\beta'})$ and $\partial B(z,\eps^\beta)$. On the right, we depict the parts of $\wt\eta$ in $B(z_1,\wh\delta_1)$, with $\wt\eta(t_1; \wh \delta_1)$ in red (which is a subset of $\varphi(e_1)$) and the other parts in grey.
Note that the definition of $\wh \delta_i$ is given in the proof of Lemma~\ref{lem:roundish}. Also note that here and in all  later figures, we depict the curve as simple for clarity, even if it is in fact allowed to be self-intersecting (but not self-crossing).}
\label{distortion}
\end{figure}

\begin{proof}
Note that for all $i\in[1,N]$, since $\varphi(e_i)$ has diameter $\delta_i$, there exists $t_i$ so that $\wt\eta(t_i) \in \varphi(e_i)$ and $\wt\eta(t_i; \delta_i/4)\subseteq \varphi(e_i)$ (recall that the notation $\wt\eta(t_i; \delta_i/4)$ is defined in~\eqref{1}). Let $z_i:=\wt\eta(t_i).$
Let $\wt K\subseteq\h$ be some compact set that contains  $\varphi(\eta\cap K)$. 
Applying~\ref{itm:H2} to $\wt\eta$ and $\wt K$ with some $1 < \xi < \alpha$ fixed, we get that there exists $\delta_0>0$ such that for every $i$, if we let $\wh\delta_i:=\min(\delta_i/4, \delta_0)$, then
 there exists $y_i\in B(z_i,\wh\delta_i)$ such that $B(y_i,\wh\delta_{i}^{\alpha})\cap \wt{\eta} =\emptyset$ and $B(y_i,2\wh\delta_{i}^{\alpha})\cap \wt{\eta} \not=\emptyset$.  Moreover, if we let $O_i$ be the connected component of $B(z_i,\wh\delta_i)\setminus \wt{\eta}$ that contains $B(y_i, \wh\delta_i^\alpha)$, then for any point $a$ in $\partial O \setminus \varphi(e_i)$, any path contained in $O\cup\{a\}$ which connects $y_i$ to $a$ must exit the ball $B(y_i, \wh\delta_i^\xi)$.

We now show that we can choose  $\eps_0, \delta_0 > 0$ small enough so that $\varphi^{-1}(B(y_i,\wh\delta_i^\alpha))\subseteq B(z, \eps^\rho)$ for all~$i$, which will imply the lemma.  If one starts a Brownian motion from any point $w\in B(y_i,\wh\delta_i^\alpha)$ and stops it upon hitting $\wt\eta\cup\R$, then in order for it not to stop in $\varphi(e_i)$, by the previous paragraph, the Brownian motion must exit the ball $B(y_i, \wh\delta_i^\xi)$.
It follows from the Beurling estimate that the probability that the Brownian motion stops in $\varphi(e_i)$ is $1-O(\wh\delta_{i}^{(\alpha-\xi)/2})$.  Since $\varphi^{-1}$ is conformal on $\h\setminus\wt\eta$, if one starts a Brownian motion $B$ from $\varphi^{-1}(w)$ and stops it upon hitting $\eta\cup\R$, then the probability that it stops in $e_i$ is also $1-O(\wh\delta_{i}^{(\alpha-\xi)/2})$. However, if $\varphi^{-1}(w)$ is outside of $B(z,\eps^\rho)$, then by the Beurling estimate, the probability that $B$ stops in $e_i$ is $O(\eps^{(\beta-\rho)/2})$.  This is impossible as long as $\eps_0, \delta_0 > 0$ are small enough. 
\end{proof}

Lemma~\ref{lem:roundish}  implies the following lemma.
\begin{lemma}\label{lem:variation}
For any $C>0$, there exist $\eps_0, \delta_0\in(0,1)$, such that for any $z\in K$ and $\eps \in (0,\eps_0)$ with $\eta \cap B(z,\eps) \neq \emptyset$, for any $\delta\le C\min(\wh\delta, \delta_0)$, 
there exist $y\in\delta^\alpha\Z^2$ such that  
\begin{align*}
B(y, 2\delta^\alpha)\cap \wt{\eta} \neq \emptyset\quad \text{and}\quad B(y, \delta^\alpha) \subseteq \varphi(B(z, \eps^\rho)).
\end{align*}
\end{lemma}
\begin{proof}
Lemma~\ref{lem:roundish} implies that for $\eps_0, \delta_0$ small enough,  $\varphi(B(z,\eps^\rho))$ contains some $B(y, \min(\wh\delta, \delta_0)^\alpha)$ such that~\eqref{byd} is satisfied.

For any $\alpha'>\alpha$, one can always make $\delta_0$ small enough, so that  for any $\delta\le C\min(\wh\delta, \delta_0)$, we have $\delta^{\alpha'}\le \min(\wh\delta, \delta_0)^\alpha$. In this case, $\varphi(B(z,\eps^\rho))$  must also contain some ball $B(y', \delta^{\alpha'})$ where $y'\in\delta^{\alpha'} \Z^2$ and $B(y', 2\delta^{\alpha'})\cap \wt\eta\not=\emptyset$.
This proves the present lemma with $\alpha'$ instead of $\alpha$. However, since $\alpha$ is an arbitrary number in $(1,\infty)$, so is $\alpha'$, hence we are done.
\end{proof}

In the following lemma, we will compare the diameters $\delta_i$ of the excursions to the diameter of $\varphi(B(z,\eps))$, which will later allow us to apply Lemma~\ref{lem:variation} for $\delta=\diam(\varphi(B(z,\eps)))$.

\begin{lemma}\label{lem:order}
There exist $C_0>0$ and $\eps_0\in(0,1)$ such that, for any $z\in K$ and $\eps \in (0,\eps_0)$ with $\eta \cap B(z,\eps) \neq \emptyset$, for $C=M(2 C_0+1)$, we have
\begin{align}
\label{eq:delta}
\diam(\varphi(B(z,\eps)))\le (2C_0+1) \sum_{i=1}^N \delta_i \le C \wh\delta.
\end{align}
\end{lemma}
\begin{proof}
We would like to show that one can choose $C_0 > 0$ big enough and $\eps_0 > 0$ small enough, such that for all $\eps \in(0,\eps_0)$ with $\eta \cap B(z,\eps) \neq \emptyset$, if $D$ is any connected component of $B(z, \eps^\beta)\setminus \eta$ that intersects $B(z, \eps)$, then we have
\begin{align}\label{eq:inclusion}
\varphi\left(D \cap B(z, \eps)\right)\subseteq \bigcup_{i=1}^N B\!\left(\varphi(e_i), C_0 \delta_i\right),
\end{align}
where $B\!\left(\varphi(e_i), C_0 \delta_i\right)$ denotes the $C_0 \delta_i$ neighborhood of the set $\varphi(e_i)$.
If this is true for any such $D$, then we would have proven that
$\varphi(B(z, \eps))$ is included in the closure of the right hand side of~\eqref{eq:inclusion}.  
Note that each of the $B(\varphi(e_i),C_0 \delta_i)$ has diameter at most $(2C_0+1)\delta_i$, hence any connected component of the closure of the right hand side of~\eqref{eq:inclusion} has diameter at most $\sum_{i=1}^N (2C_0+1)\delta_i$. Since $\varphi(B(z,\eps))$ is connected, we see that~\eqref{eq:inclusion} implies that~\eqref{eq:delta} holds.

Now, let $D$ be a connected component of $B(z, \eps^\beta)\setminus \eta$ that intersects $B(z, \eps)$.  For any point $y\in D \cap B(z, \eps)$, the Beurling estimate implies that a Brownian motion started from $y$ and stopped upon hitting $\eta\cup \R$ hits $\bigcup_{i=1}^N  e_i $ with probability $1-O(\eps^{(1-\beta')/2})$. The map $\varphi$ is conformal on $D$, hence a Brownian motion $B$ started from $\varphi(y)$ and stopped upon hitting $\wt\eta\cup \R$ also hits $\bigcup_{i=1}^N \varphi(e_i)$ with probability $1-O(\eps^{(1-\beta')/2})$.  If $\varphi(y)$ is outside of $B(\varphi(e_i), C_0 \delta_i)$, then the Beurling estimate implies that the probability that $B$ ends at $\varphi(e_i)$ is smaller than $cC_0^{-1/2}$ where $c > 0$ is some absolute constant.  Hence if $\varphi(y)$ is outside of $\bigcup_{i=1}^N B(\varphi(e_i), C_0 \delta_{i})$, then the probability that $B$ stops in $\bigcup_{i=1}^N \varphi( e_i)$ is smaller than $N c C_0^{-1/2} \leq Mc C_0^{-1/2}$.  If we choose $C_0$ big enough and $\eps_0$ small enough, then this is impossible, hence~\eqref{eq:inclusion} is true, so we are done.
\end{proof}

\subsection{Proof of Proposition~\ref{prop:distortion}}
\label{sec:proof_of_proposition}

We are now ready to prove Proposition~\ref{prop:distortion}.
Recall that $1 \leq d<2$ is the upper Minkowski dimension of $\wt\eta$.
Choose some $r_0\in (1, 1/(d-1))$ (where we take the interval to be $(1,\infty)$ in the case $d=1$). 
Let $\delta_0 \in(0,1)$ be the one chosen in Lemma~\ref{lem:roundish}.
For any $r \in(0, r_0)$, let
\begin{align*}
\CE(r,\eps):=\{ z\in K: B(z,\eps)\cap \eta\not=\emptyset,\, \diam(\varphi(B(z,\eps)))\in (\eps^r, 2\eps^r] \}.
\end{align*}
We also define
\begin{align*}
&\CF(r_0, \eps):=\{ z\in K: B(z,\eps)\cap \eta\not=\emptyset,\, \diam(\varphi(B(z,\eps)))\le \eps^{r_0}\}, \\
&\CG(\delta_0, \eps):= \{ z\in K: B(z,\eps)\cap \eta\not=\emptyset,\, \diam(\varphi(B(z,\eps)))\ge \delta_0 \}.
\end{align*}
If we take $r_n:=r_0-n \log 2/ \log \eps^{-1}$ for all $0\le n\le m$ where $m=r_0 \log \eps^{-1} /\log 2$, then we have
\begin{align*}
\{ z\in K: B(z,\eps)\cap \eta\not=\emptyset \}=\CG(\delta_0, \eps) \cup \CF(r_0, \eps)\cup \bigcup_{n=0}^m \CE(r_n,\eps).
\end{align*}

\begin{lemma}\label{lem:set1}
For any $\iota>0$, the area of  $\CG(\delta_0,\eps)$ is $O(\eps^{2-\iota})$.
\end{lemma}
\begin{proof}
For any $z\in \CG(\delta_0,\eps)$,   $ \diam(\varphi(B(z,\eps))) \ge \delta_0 $. We also know by Lemma~\ref{lem:order} that $\diam(\varphi(B(z,\eps))) \le C \wh\delta$. This implies $\delta_0 \le C\min( \wh\delta, \delta_0)$.
We can therefore apply Lemma~\ref{lem:variation} and deduce that $\varphi\left(B(z,\eps^\rho)\right)$ contains some ball  $B\left(y, \delta_0^\alpha\right)$
where $y$ belongs to the following set
\begin{align*}
Y:=\left\{ y\in \!\left(\delta_0^{ \alpha} \Z^2\right) \cap \varphi(K): B(y, 2\delta_0^\alpha) \cap \wt\eta\not= \emptyset \right\}.
\end{align*}
Therefore, the union of the balls $B(\varphi^{-1}(y), \eps^\rho)$ for all $y\in Y$ covers $\CG(\delta_0,\eps)$.
Since $\wt\eta$ has upper Minkowski dimension $d$, we have $|Y|=O\!\left((\delta_0^\alpha)^{-d-\iota'}\right)$ for any $\iota'>0$.
The area of  $\CG(\delta_0,\eps)$ is therefore at most $\pi\eps^{2\rho}  |Y| =\eps^{2\rho} O\!\left((\delta_0^\alpha)^{-d-\iota'}\right)$. Since $\rho, \alpha$ are arbitrarily close to $1$ and $\iota, \iota'$ are  arbitrarily close to $0$, we get the bound in the lemma.
\end{proof}

\begin{lemma}
\label{lem:set2}
For any $r \in(0,r_0)$ and any $\iota>0$, the area of  $\CE(r,\eps)$ is $O(\eps^{2-rd-\iota})$.
\end{lemma}
\begin{proof}
We already know that the  area of  $\CE(r,\eps) \cap \CG(\delta_0,\eps)$ is $O(\eps^{2\rho})$ for any $\rho<1$. Hence we only need to consider the case $ \diam(\varphi(B(z,\eps))) < \delta_0$.
Lemma~\ref{lem:order} implies that $\diam(\varphi(B(z,\eps))) \le C \wh\delta$. Therefore we have that $ \eps^r\le\diam(\varphi(B(z,\eps))) \le C \min(\wh\delta, \delta_0)$.
We can therefore apply Lemma~\ref{lem:variation} and deduce that $\varphi\left(B(z,\eps^\rho)\right)$ contains some ball  $B\left(y, \eps^{r\alpha}\right)$ where $y$ belongs to the following set
\begin{align*}
Y_r:=\left\{ y\in \!\left(\eps^{r \alpha}\Z^2\right) \cap \varphi(K): B(y, 2\eps^{r\alpha}) \cap \wt\eta\not= \emptyset \right\}.
\end{align*}
Therefore, the union of the balls $B(\varphi^{-1}(y), \eps^\rho)$ for all $y \in Y_r$ covers $\CE(r,\eps)\setminus  \CG(\delta_0,\eps)$. 
Since $\wt\eta$ has upper Minkowski dimension $d$, we have $|Y_r|=O\!\left((\eps^{r\alpha})^{-d-\iota'}\right)$ for any $\iota'>0$.
The  area of  $\CE(r,\eps)$ is therefore at most $\pi\eps^{2\rho} |Y_r|= \eps^{2\rho} O\!\left((\eps^{r\alpha})^{-d-\iota'}\right)$. Since we can choose $\rho, \alpha$ arbitrarily close to $1$ and $\iota, \iota'$ arbitrarily close to $0$, we get the bound in the lemma.
\end{proof}

For $z$ uniformly chosen in $K$, we can compute the following expectation:
\begin{align*}
&\E\!\left[\diam\!\left(\varphi(B(z,\eps))\right) \mathbf{1}_{d(z,\eta)<\eps} \right]\\
\le & \sum_{n=0}^m \p\!\left[z\in \CE(r_n, \eps)  \right] 2 \eps^{r_n} +  \p\!\left[z\in \CF(r_0, \eps) \right] \eps^{r_0} + \p\!\left[z\in \CG(\delta_0, \eps) \right] \diam(\varphi(K)). 
\end{align*}
Applying Lemma~\ref{lem:set2} to bound the probabilities in the sum above, Lemma~\ref{lem:set1} to bound the probability in the last term above, and using the trivial bound of $1$ for the probability in the middle term above, we see that
\begin{align}\label{eq:exp3}
\E\!\left[\diam\!\left(\varphi(B(z,\eps))\right) \mathbf{1}_{d(z,\eta)<\eps} \right] \leq \sum_{n=0}^m O( \eps^{2-r_n d - \iota}) \, 2\eps^{r_n}+ \eps^{r_0} + O( \eps^{2-\iota}).
\end{align}
Note that $\eps^{r_n}=\eps^{r_0} 2^n$, hence the right hand side of~\eqref{eq:exp3} is equal to
\begin{align}\label{bound1}
 O\!\left( \eps^{2-\iota} \right) \sum_{n=0}^m   (\eps^{r_0} 2^n)^{1-d}  + \eps^{r_0} + O( \eps^{2-\iota}).
\end{align}
If $d=1$, then choose $r_0=2$. For any $\iota' >\iota$, \eqref{bound1} is at most
\begin{align*}
m\, O(\eps^{2-\iota})  + \eps^{r_0} + O( \eps^{2-\iota}) =O((\log \eps^{-1}) \eps^{2-\iota})= O(\eps^{2-\iota'})=O(\eps^{2/d-\iota'}).
\end{align*}
 Note that $\iota, \iota'$ can be chosen arbitrarily close to $0$, hence the above equation proves Proposition~\ref{prop:distortion} for $d=1$.
 
Otherwise if $d\in(1,2)$, then choose $r_0=2/d$, which is in the interval $(1, 1/(d-1))$. Then~\eqref{bound1} is equal to 
\begin{align*}
O\!\left(\eps^{2/d-\iota} \right).
\end{align*}
This completes the proof of Proposition~\ref{prop:distortion}.  \qed

\section{Checking the hypotheses for SLE} \label{sec:non_tracing}
In this section, we fix $\kappa\in(0,8)$ and let $\eta$ be an  $\SLE_\kappa$ curve in $\h$ from $0$ to $\infty$. By definition, $\eta$ is non-self-crossing. By \cite{rs2005basic}, we have that $\eta$ a.s.\ has upper Minkowski dimension at most $d$ for any $d > 1+\kappa/8\in(1,2)$ and zero Lebesgue measure.  (In fact, by \cite{lr2015minkowski}, $\eta$ a.s.\ has Minkowski dimension $1+\kappa/8$, but we will not need this stronger result.)  We will show that $\eta$ a.s.\ satisfies~\ref{itm:H1} and~\ref{itm:H2}.

\subsection{Hypothesis~\ref{itm:H1}}
The following lemma says that $\eta$ a.s.\ satisfies~\ref{itm:H1}.
\begin{lemma}
\label{lem:crossing}
For each $\beta \in (0,1)$ and compact rectangle $K \subseteq \h$, there a.s.\ exist $M>0$ and $\eps_0>0$, such that for all $\eps \in (0,\eps_0)$, and for all $z\in K$, the number of excursions  of $\eta$ between $\partial B(z, \eps^{\beta})$ and $\partial B(z,\eps)$ is at most $M$.\end{lemma}
\begin{proof}
For any fixed $z$, the probability that $\eta$ makes $k$ excursions between $\partial B(z, 4\eps)$ and $\partial B(z,\eps^\beta/2)$  decays exponentially. A rough upper bound of this probability can be found in \cite[Theorem~5.7]{MR2981906}, which is
\begin{align}\label{eq:pcrossing}
O\!\left( \eps^{c_0 (1-\beta) k} \right),
\end{align}
where $c_0>0$ is some constant depending only on $\kappa$.
One can find $k=M$ such that~\eqref{eq:pcrossing} is  $O(\eps^4)$.

We can now apply the Borel-Cantelli arguments. Let $\eps_n=1/n$. Let $F_n$ be the event that there exists  $z\in K \cap \eps_n \Z^2$ so that there are more than $M$ excursions between  $\partial B(z, 4\eps_n)$ and $\partial B(z,\eps^\beta_n/2)$. 
By the union bound, the probability of $F_n$ is $O(\eps_n^2)$, which is summable in $n$. This implies that
there a.s.\ exists $n_0 \in \N$ such that for all $n \geq n_0$, and all $z\in K \cap \eps_n \Z^2$, $\eta$ makes no more than $M$ excursions between  $\partial B(z, 4\eps_n)$ and $\partial B(z,\eps^\beta_n/2)$. 

We can now pick $\eps_0:=\eps_{n_0}/2$.  For each $\eps \in (0,\eps_0)$, one can find $n \geq n_0$, such that $\eps_{n+1}\le \eps <\eps_{n}$.
For all $z\in K$, there exist  $z_0\in \eps_n \Z^2$ such that $ B(z, \eps)\subseteq B(z_0, 4\eps_n)$ and $B(z,\eps^\beta)\supset B(z_0, \eps^\beta_n/2)$. The number of crossings between  $\partial B(z, \eps^{\beta})$ and $\partial B(z,\eps)$ is therefore at most the number of crossings between  $\partial B(z_0, 4\eps_n)$ and $\partial B(z_0,\eps^\beta_n/2)$, which is at most $M$.
\end{proof}

\subsection{Hypothesis~\ref{itm:H2}}\label{sec4.2}

In this section, our goal is to show that $\eta$ a.s.\ satisfies~\ref{itm:H2}. In Section~\ref{subsec1}, we will first reduce the proof of~\ref{itm:H2} to that  of Proposition~\ref{lem:contains_ball} and then further boil it down to the proof of Lemma~\ref{lem:borel-cantelli}. In Section~\ref{subsec2}, we will focus on proving Lemma~\ref{lem:borel-cantelli}.

\subsubsection{Outline of the proof}\label{subsec1}
In order to prove that $\eta$ a.s.\ satisfies~\ref{itm:H2}, we will show that it is enough to prove Proposition~\ref{lem:contains_ball}.
Heuristically speaking, \ref{itm:H2} says that one can find a ball of size $\delta^\alpha$ near an excursion $\eta(t;\delta)$ which is in a certain sense far away from the other parts of $\eta$ in $B(\eta(t), \delta)$. In Proposition~\ref{lem:contains_ball}, we show that one can find a small ball near $\eta(t;\delta)$ which is shielded from the other parts of $\eta$ by a well-chosen arc.

Throughout, we shall assume that we have fixed a compact rectangle $K$ and the parameters $\alpha>\gamma>\xi>\lambda>1$, $\mu>1$.
We also introduce the following notation: For $t>0$ and $\delta>0$, the excursion $\eta(t; \delta)$ is defined in~\eqref{1}. For any excursion $e$ of the type $\eta(t; \delta)$, let  $B(w,r)$ be a ball that intersects~$e$.  For each $y\in B(w,r) \setminus e$, first note that, the intersection of $\partial B(w,r)$ with the boundary of the connected component of $B(w,r)\setminus e$ containing $y$ is a closed arc, then let $A\!\left(e, w,r,y\right)$ be the open arc obtained from taking away the two endpoints of this closed arc. See Figure~\ref{find-ball}.  We can now state Proposition~\ref{lem:contains_ball}.

\begin{proposition}
\label{lem:contains_ball}
There a.s.\ exists $\delta_0>0$ such that for any $\delta \in (0,\delta_0)$, for any $t>0$ such that $\eta(t) \in K$, one can find $w \in B(\eta(t),\delta/2)$ and $y$ that satisfy the following condition:
\begin{equation}\label{condition1}
 B(y,\delta^{\alpha}) \subseteq B(w, \delta^{\gamma})\setminus\eta, \,
B(y, 2\delta^{\alpha} ) \cap\eta\not=\emptyset, \,
 B(w, \delta^\gamma)\cap \eta(t;\delta)\not=\emptyset,
 \, A(\eta(t;\delta), w, \delta^{\xi}, y )\cap\eta=\emptyset.
\end{equation}
\end{proposition}

\begin{figure}[h!]
\centering
\includegraphics[scale=0.85]{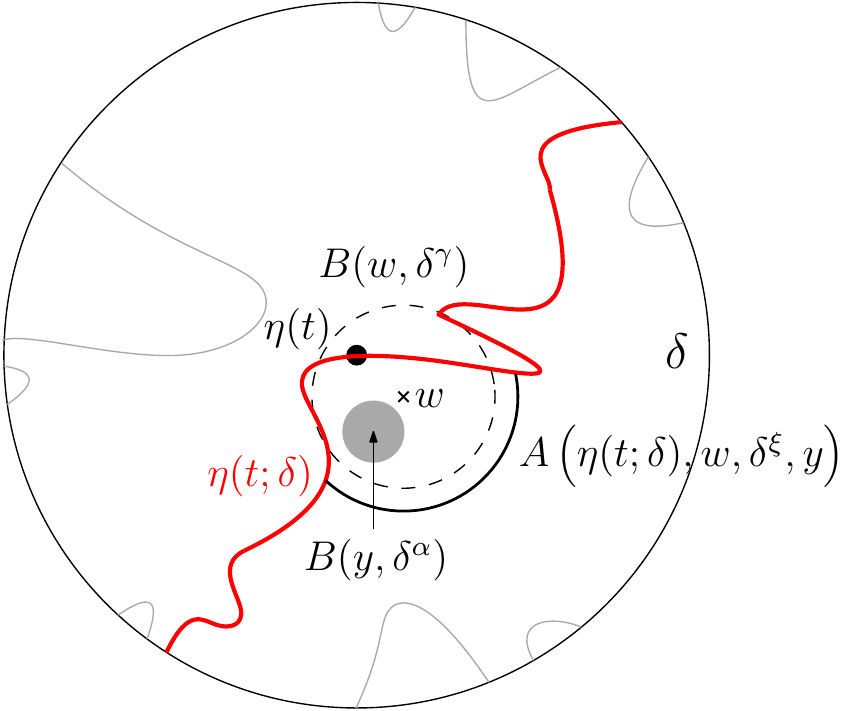}
\caption{Illustration of Proposition~\ref{lem:contains_ball}. The bold black arc represents $A\!\left(\eta(t;\delta), w,\delta^\xi,y\right)$.}
\label{find-ball}
\end{figure}

\begin{lemma}\label{lem:imply_H2}
If Proposition~\ref{lem:contains_ball} holds, then $\eta$ satisfies~\ref{itm:H2}.
\end{lemma}
\begin{proof}
Let $y$ be a point chosen according to Proposition~\ref{lem:contains_ball}.  Then it immediately satisfies condition (i) of~\ref{itm:H2}. It then remains to check that it also satisfies condition (ii) of~\ref{itm:H2}.
 Let $O$ denote the connected component of $B(\eta(t), \delta) \setminus \eta$ which contains $y$.
Fix $\xi'\in(\xi, \alpha)$.
We want to check that, for any point $a$ in $\partial O\setminus \eta(t;\delta)$, any path from $y$ to $a$ which is contained in $O\cup\{a\}$ must exit $B(y, \delta^{\xi'})$. This statement exactly describes the condition (ii) of~\ref{itm:H2}, because it in fact holds for arbitrary $\alpha>\xi'>1$, since $\alpha>\xi>1$ can be chosen arbitrarily.
By the last condition of~\eqref{condition1}, such a path must cross $A(e,w,\delta^\xi, y)$. Moreover, since $d(y,w)< \delta^\gamma$, it follows that this path also exits $B(y, \delta^{\xi'})$ for some $\xi'\in(\xi, \alpha)$. This completes the proof.
\end{proof}

The proof of Proposition~\ref{lem:contains_ball} builds on  the following key lemma. Assuming Lemma~\ref{lem:borel-cantelli}, we can conclude using Borel-Cantelli arguments.

\begin{lemma}
\label{lem:borel-cantelli}
Let $E(z,\delta)$ be the event that $\eta \cap B(z, \delta^\mu) \neq \emptyset$ and for any excursion $e$ of $\eta$ between $\partial B(z, \delta^\mu)$ and $\partial B(z, \delta)$, there exist $w\in B(z,\delta/2)$ and $y$ that satisfy the following condition:
\begin{equation}\label{condition2}
B(y,\delta^\alpha)\subseteq B(w, \delta^\gamma)\setminus\eta,\, B(y,2\delta^\alpha)\cap\eta\not=\emptyset, \, B(w,\delta^\gamma)\cap e\not=\emptyset, \, A(e, w,\delta^\xi, y)\cap\eta=\emptyset.
\end{equation}
For every $z \in K$, we have that $E(z,\delta)$ holds with probability $1-O(\delta^4)$.
\end{lemma}

\begin{proof}[Proof of Proposition~\ref{lem:contains_ball}]

Let $\delta_n:=1/n$. We define $F_n$ to be the event that there exists  $z\in K\cap \delta_n^\mu \Z^2$ such that $E(z,\delta_n)$ does not hold.  By Lemma~\ref{lem:borel-cantelli} and a union bound for all  $z\in K\cap \delta_n^\mu \Z^2$, we get that 
\[ \p[F_n] = O(\delta_n^{-2\mu} \times \delta_n^{4}).\]
The quantity above is summable whenever $\mu > 1$ is sufficiently close to $1$.
By the Borel-Cantelli lemma, we know that there a.s.\ exists $n_0 \in \N$ such that for all $n \geq n_0$, $F_n$ does not occur.
This means that for all $n \geq n_0$ and all $z\in K \cap \delta_n^\mu \Z^2$ such that $\eta$ intersects $B(z, \delta_n^\mu)$, for any excursion $e$ of $\eta$ between $\partial B(z, \delta_n^\mu)$ and $\partial B(z,\delta_n)$, one can find $w\in B(z,\delta_n)$ and $y$ that satisfy the condition~\eqref{condition2} for $e$ and $\delta=\delta_n$.

For any $\delta \in (0,\delta_{n_0})$, we can find $n \geq n_0$ such that $\delta_{n}\le \delta/4 \le \delta_{n-1}$. For any $t>0$ such that $\eta(t) \in K$, there must exist $z\in K \cap \delta_{n}^\mu \Z^2$ such that $\eta(t; \delta)$ contains an excursion $e$  between $B(z, \delta_{n}^\mu)$ and $B(z, \delta_{n})$.
We know that there exist $w\in B(z,\delta_n)$ and $y$ such that $B(y,\delta_n^\alpha)\subseteq B(w, \delta_n^\gamma)\setminus \eta$, $B(y,2\delta_n^\alpha)\cap \eta\not=\emptyset, B(w,\delta_n^\gamma)\cap e\not=\emptyset$ and $ A(e, w,\delta_n^\xi, y)\cap\eta=\emptyset$.
This last condition implies that 
\[ A(\eta(t;\delta), w,\delta_n^\xi, y) = A(e, w,\delta_n^\xi, y).\]
It is then not difficult to change $\delta_n$ into $\delta$ in the statement (since this is a standard step similar to what we did in the proof of the previous lemma, we omit it here). We can then conclude.
\end{proof}

\subsubsection{Proof of  Lemma~\ref{lem:borel-cantelli}}\label{subsec2}
This section is dedicated to the proof of Lemma~\ref{lem:borel-cantelli}.
We will first prove Lemma~\ref{lem:greens_function} and deduce that for $\delta >0$ small enough, for any ball $B(w,\delta^\gamma)$, one can a.s.\ find a ball $B(y, \delta^\alpha)$ contained in $B(w,\delta^\gamma)\setminus\eta$ such that $B(y,2\delta^\alpha)\cap \eta\not=\emptyset$.
Then we will need to find such a pair $(w,y)$ which satisfies the additional condition that $ B(w, \delta^\gamma)$ intersects $\eta(t;\delta)$ and $A\!\left(\eta(t;\delta), w, \delta^{\xi}, y\right)$ does not intersect $\eta$.

\begin{lemma}
\label{lem:greens_function}
There a.s.\ exists $\delta_0>0$ such that for any $\delta \in (0,\delta_0)$ and any $w\in K$, there exists $y$ such that 
\[ B(y,\delta^\alpha)\subseteq B(w,\delta^\gamma)\setminus\eta. \]
\end{lemma}

We remark that as a consequence of Lemma~\ref{lem:greens_function}, for any $w\in K$ such that $B(w,\delta^\gamma)\cap \eta\neq\emptyset$, knowing that we can find a ball $B(y,\delta^\alpha)$ contained in $B(w,\delta^\gamma)\setminus\eta$,  we can then also move the ball $B(y,\delta^\alpha)$ inside $B(w,\delta^\gamma)\setminus\eta$ so that we also have
$B(y,\delta^\alpha)\subseteq B(w,\delta^\gamma)\setminus\eta$ and $B(y, 2\delta^\alpha) \cap \eta\not=\emptyset. $

\begin{proof}
Fix $\zeta\in(\gamma, \alpha)$.
For any $m \in \N$, if $\delta >0$ is small enough, then for any $w \in K$ we can place $m$ balls of radius $2\delta^\alpha$ in $B(w, \delta^\gamma/4)$ such that their mutual distances are greater than $\delta^\zeta$. Let $\wt E(w,\delta)$ be the event that $\eta$ intersects all of these $m$ balls.  Using the $n$-point Green's function for chordal SLE (see \cite[Proposition~2.3]{rezaei2016green}), we get that the probability of $\wt E(w,\delta)$ is at most an absolute constant times $\delta^{m(2-d)(\alpha-\zeta)}$.  We then choose $m$ big enough so that $m(2-d)(\alpha-\zeta) +2\gamma >2$.

We will now use the Borel-Cantelli lemma to complete the proof.  Let $\delta_n:=1/n$. Let $\wt F_n$ be the event that there exists $w\in K \cap (\delta_n^{\gamma}/4) \Z^2$ such that the event $\wt E(w, \delta_n)$ holds.  By the union bound, the probability of $\wt F_n$ is at most a constant times $n^{-m(2-d)(\alpha-\zeta)+2\gamma}$ which is at most a constant times~$n^{-2}$. The sequence~$n^{-2}$ is summable, hence there a.s.\ exists $n_0 \in \N$ such that for all $n\ge n_0$, $\wt F_n$ does not occur.  That is, for all $n\ge n_0$ and $w\in K \cap (\delta_n^{\gamma}/4) \Z^2$, there exists $B(y, 2\delta_n^{\alpha})$ which is contained in $B(w, \delta_n^{\gamma}/4)$ and does not intersect $\eta$.  This implies that for all $\delta \in (0,\delta_{n_0})$ and $w\in K$, the ball $B(w, \delta^\gamma)$ a.s.\ contains some ball $B(y, \delta^\alpha)$ that does not intersect $\eta$.  

\end{proof}

The main idea in proving Lemma~\ref{lem:borel-cantelli} is the following:
Given an excursion $\eta(t; \delta)$, we place a number of small balls near this excursion which are respectively shielded by disjoint arcs (thanks to Lemma~\ref{lem:greens_function}). Then we will show that the future and past (w.r.t.\ time $t$) parts of $\eta$ have a very small probability to hit all of the shielding arcs. We first establish this result for the future part of $\eta$ and then show that it simultaneously holds for the past of $\eta$, using reversibility of $\SLE$ \cite{zhan2008reverse,ms2016imag3}. (We expect that one can prove this result without using reversibility, but reversibility simplifies the proof.)
More concretely, we will make use of the spatial Markov property of SLE and rely on fine estimates of the deformation of the relevant arcs under conformal maps.

Let us now start to prove Lemma~\ref{lem:borel-cantelli}. We will prove a series of lemmas, and the proof of Lemma~\ref{lem:borel-cantelli} will be completed at the very end of the paper. Suppose that $\eta$ makes $N$ excursions between $\partial B(z, \delta^\mu)$ and $\partial B(z, \delta)$ and we denote them by $e_1,\ldots, e_N$ (in chronological order).  For each $i$, let $s_i<t_i$ be such that $e_i=\eta[s_i, t_i]$. Due to the spatial Markov property of SLE, each $t_i$ is a stopping time for the filtration $\CF_t$ generated by $\eta|_{[0,t]}$.

Choose some constant $L$ such that
\begin{align}\label{eq:L}
L (8/\kappa-1) (\xi-\lambda)/4 >4.
\end{align}
Let $R:=L^{L+1}$.
For each $i\in[1,N]$, we will place a number of balls with radius $\delta^\xi$ centered at points in $W_i$, where  $W_i$ is defined in the following way:
Let $\wt e_i:=e_i \cap  B(z, \delta/2)$.  We want to choose $W_i$ to be a set of $2R$ points that are all in $\wt e_i$ with mutual distances at least $\delta^\lambda$. We also want to choose $W_i$ in a measurable way w.r.t.\ $e_i$ as a set (i.e., without the time parameterization). We first choose $w_1\in \wt e_i$ to be the leftmost point of $\wt e_i$, breaking ties by taking the point with the smallest $y$-coordinate.  For any $j\in[1,2R-1]$, assume that we have chosen the first $j$ points, we will choose $w_{j+1}$ to be the leftmost point in $\wt e_i\setminus \bigcup_{k=1}^j B(w_k, \delta^\lambda)$, breaking ties by taking the point with the smallest $y$-coordinate. 
Note that $\wt e_i\setminus \bigcup_{k=1}^j B(w_k, \delta^\lambda)$ is non-empty, because $\wt e_i$ is a connected set with diameter $\delta$ and each of the connected components of $\bigcup_{k=1}^j B(w_k, \delta^\lambda)$ has diameter at most $4R \delta^\lambda$ which is less than $\delta$ when $\delta >0$ is small enough. We have thus defined $W_i$.

For any $w\in W_i$, Lemma~\ref{lem:greens_function} implies that there exists some point $y$ such that 
\begin{align*}
B(y,\delta^\alpha)\subseteq B(w,\delta^\gamma)\setminus\eta.
\end{align*}
First note that since $e_i\subseteq \eta$, we have $B(y,\delta^\alpha)\subseteq B(w,\delta^\gamma)\setminus e_i$.
Since $w\in e_i$, we can then move the ball $B(y, \delta^\alpha)$ inside $B(w,\delta^\gamma)\setminus e_i$ so that
\begin{align}\label{cond_yw}
B(y,\delta^\alpha)\subseteq B(w,\delta^\gamma)\setminus e_i, \, \overline{B(y,3\delta^\alpha/2)} \cap e_i \not=\emptyset.
\end{align}
The set of $y$ satisfying~\eqref{cond_yw} is compact, hence we can choose $y^i(w)$ to be the leftmost $y$ that satisfies~\eqref{cond_yw}, breaking ties as above. We have thus defined $y^i$ as a function defined on $W_i$. Note that this definition of $y^i$ is measurable w.r.t.\ the set $e_i$.

Let us now list  in the following Lemma~\ref{def:W} some properties of $W_i$ and $y^i$ that follow immediately from our construction. We recall that $W_i$ and $y^i$ depend on $\delta$. Moreover, we will regard $W_i$ as a set and forget about the order of its points given by our construction.

\begin{lemma}
\label{def:W}
For all $\delta>0$ and all $i\in[1,N]$, the set $W_i$ and the function $y^i$ (defined on $W_i$) are measurable w.r.t.\ the excursion $e_i$ as a set (i.e., without the time parameterization) and there a.s.\ exists $\delta_0 > 0$ so that for all $\delta \in (0,\delta_0)$ it satisfies the following conditions:
\begin{itemize}
\item[(i)] $W_i$ contains $2R$ points and they are all in $e_i \cap B(z,\delta/2)$.

\item[(ii)] For any $w_1, w_2 \in W_i$, $\dis(w_1, w_2) \geq \delta^\lambda$.

\item[(iii)] For any $w\in W_i$, we have that  $B(y^i(w), \delta^\alpha)\subseteq B(w, \delta^\gamma)\setminus e_i$ and $B(y^i(w), 2\delta^\alpha)\cap e_i\not=\emptyset$. 
\end{itemize}
\end{lemma}

To prove Lemma~\ref{lem:borel-cantelli}, it now only remains to find at least one $w\in W_i$ for each $i$ such that $A\left(e_i,w,\delta^\xi, y^i(w)\right)$ does not intersect $\eta$.
Our first goal is to prove the following lemma.

\begin{lemma}
\label{lem:S}
For all $i\in[1,N]$, conditionally on $\{t_i<\infty\}$ and $\CF_{t_i}$, with conditional probability $1- O(\delta^4)$ (here and in the sequel, by $O(\delta^4)$, we mean that this term is a.s.\ bounded by $C \delta^4$ for some deterministic constant $C$ which is independent of  $\delta$ or $i$ but possibly dependent on $R$), one can find $S_i\subseteq W_i$ such that $|S_i|\ge R+1$ and  for all $w\in S_i$, $\eta|_{[t_i,\infty)}$ does not intersect $A\!\left(e_i,w,\delta^\xi, y^i(w)\right)$.
\end{lemma}
\begin{proof}
To prove Lemma~\ref{lem:S}, it suffices to prove that conditionally on $\CF_{t_i}$, for any $R$ points $w_1,\ldots, w_R \in W_i$, the probability that $\eta|_{[t_i,\infty)}$ intersects $A\!\left(e_i,w_j,\delta^\xi, y^i(w_j)\right)$ for all $j\in[1,R]$ is $O(\delta^4)$.  This will imply that conditionally on $\CF_{t_i}$, the probability of not finding any set $S_i$ as required in  Lemma~\ref{lem:S} is at most
\begin{align*}
\binom{2R}{R} O(\delta^4) = O(\delta^4).
\end{align*}
Hence it will imply Lemma~\ref{lem:S}.

Let $J_i$ be the Loewner hull of $\eta[0,t_i]$, i.e.,  the complement of the infinite connected component of $\h\setminus \eta[0,t_i]$.
Let $g_{t_i}$ be a conformal map from $\h\setminus J_i$ onto $\h$ that fixes $\infty$ (such a conformal map is not unique because there is still the freedom of scaling and translation, but we just choose any one of them).
For any $w\in W_i$, if $A\!\left(e_i,w,\delta^\xi, y^i(w)\right)$ is entirely contained in $J_i$ (this is possible whenever $\kappa>4$ for which case $\eta$ is self-touching), then  it  cannot intersect $\eta[t_i,\infty)$.  
Assume that there are at least $R$ points in $W_i$ whose associated $A\!\left(e_i,w,\delta^\xi, y^i(w)\right)$ arcs are not entirely contained in $J_i$ (otherwise we would have already found the set $S_i$ as required in Lemma~\ref{lem:S}). 

For any $w\in W_i$ such that $A\!\left(e_i,w,\delta^\xi, y^i(w)\right)$ is not entirely contained in $J_i$, let $A_i(w)$ denote the arc which is the intersection between $A\!\left(e_i,w,\delta^\xi, y^i(w)\right)$ and the boundary of the connected component of $\h\setminus \left(J_i \cup A\!\left(e_i,w,\delta^\xi, y^i(w)\right) \right)$ containing $\infty$. 
Note that $\eta|_{[t_i,\infty)}$ intersects $A\!\left(e_i,w,\delta^\xi, y^i(w)\right)$ if and only if it intersects $A_i(w)$.
Each $A_i(w)$ is mapped by $g_{t_i}$ to a deformed arc attached to the real line. See Figures~\ref{non-tracing} and~\ref{non-tracing2}.
Since we have assumed that there are at least $R$ such points in $W_i$, they will get mapped by $g_{t_i}$ to $R$ arcs attached to the real line which are in addition disjoint, hence either one next to another or one under another.  The nesting of the image arcs form a natural tree structure: in the image plane, if an arc is directly under another arc (there is no other arc that separates them) then the first arc is considered to be the child of the second arc. We also need to add an artificial root vertex and assign all the outermost arcs as its children. 
We denote this tree by $\CT$. 
By an abuse of language, the vertices of $\CT$ can be either points in $W_i$, their associated arcs or the images of these arcs, which will be clear in the context.
Note that $\CT$ is measurable w.r.t.\ $\CF_{t_i}$ (but not $e_i$).
This tree contains $R+1=L^{L+1}+1$ vertices, hence it must contain either a branch of depth at least $L+1$ or a vertex with at least $L$ children.
We will treat the two cases separately in the following. We will prove respectively in Lemma~\ref{lem:case1_goal} and Lemma~\ref{lem:case2_goal} that for each of the two cases, conditionally on $\CF_{t_i}$, for any $R$ points $w_1,\ldots, w_R \in W_i$, the probability that $\eta|_{[t_i,\infty)}$ intersects $A\!\left(e_i,w_j,\delta^\xi, y^i(w_j)\right)$ for all $j\in[1,R]$ is $O(\delta^4)$. 
This will complete the proof of Lemma~\ref{lem:S}.
\end{proof}

\subsubsection*{Case 1: $\CT$ contains a vertex with at least $L$ children}
Let us fix some notation which is locally used in the present case.
Let $w_0$ be the vertex which has at least $L$ children and we arbitrarily pick $L$ of its children $w_1,\ldots, w_L$. 
For $j\in[1,L]$, we denote the arcs $A_i(w_j)$ by $A_{i,j}$.
If $w_0$ is not the root of $\CT$ (i.e., an artificial vertex), then we denote  $A_i(w_0)$ by $A_{i,0}$ and let $\tau$ be the first time after $t_i$ that $\eta$ hits $A_{i,0}$.
We aim to show the following lemma.
\begin{lemma}\label{lem:case1_goal}
Conditionally on $\{t_i<\infty\}$ and  $\CF_{t_i}$, the probability that $\eta|_{[t_i, \infty)}$ visits all the arcs $A_{i,j}$ for $j\in[1,L]$ is $O\!\left(\delta^{L (8/\kappa-1) (\xi-\lambda)/4}\right)$.  In particular, if $L$ satisfies~\eqref{eq:L}, then this probability is $O(\delta^4)$.
 \end{lemma}
 
We will prove Lemma~\ref{lem:case1_goal} for the case where $w_0$ is not the root of $\CT$. When $w_0$ is the root of $\mathcal{T}$, it is an artificial (virtual) vertex and its children are all the outermost arcs. The proof for this case will follow from almost the same arguments with $t_i$ in the place of $\tau$, hence we decide to leave it to the reader.
 
 \begin{figure}[h!]
\centering
\includegraphics[scale=0.85]{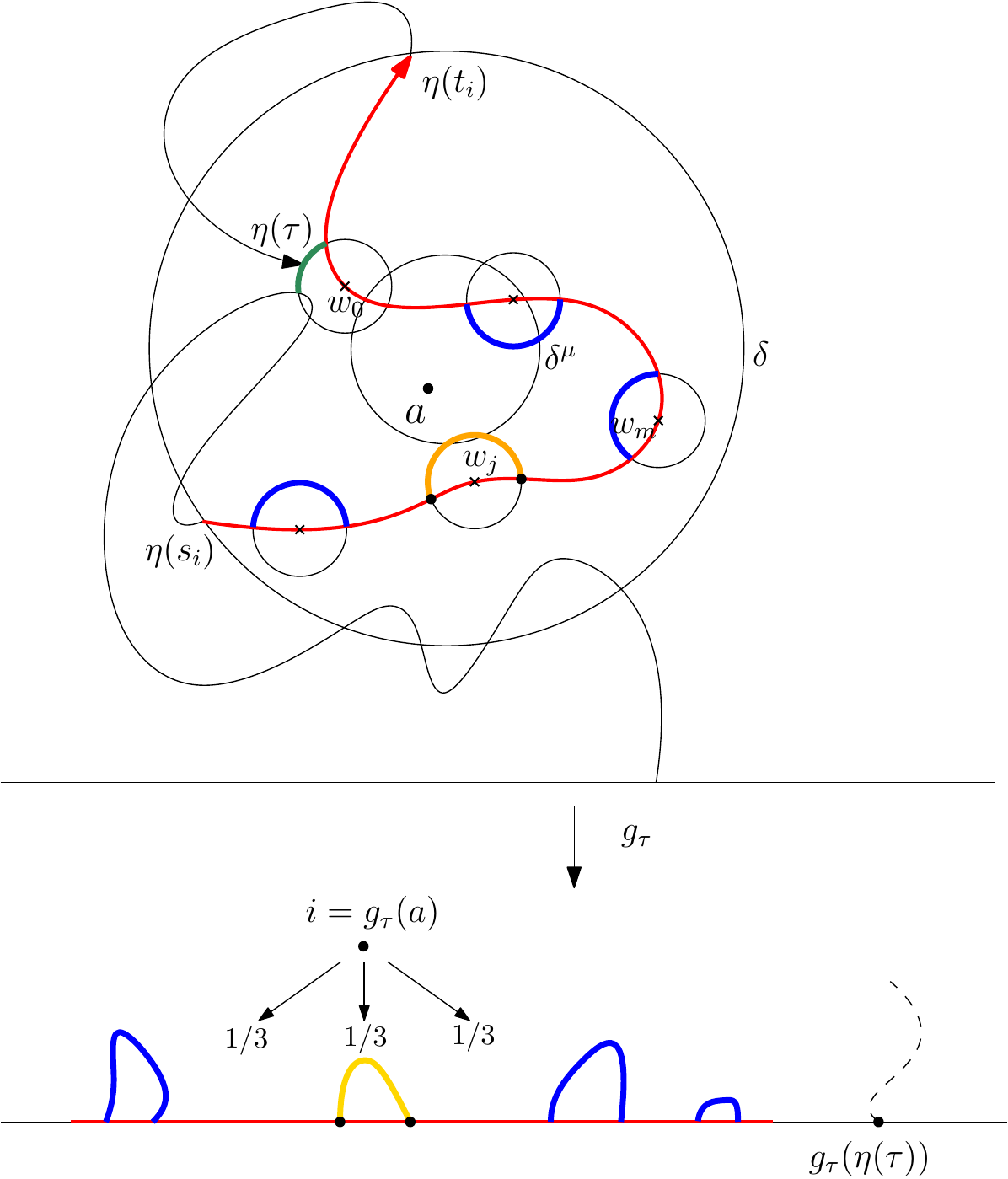}
\caption{Illustration of Case 1. On the top, we depict the excursion $e_i$ (in red) and the balls $B(w,\delta^\xi)$ where $w\in W_i$. The green arc represents $A_{i,0}$. The blue and yellow arcs are the children of $A_{i,0}$. The yellow arc centered at  $w_{j}$ is chosen for the renormalization of $g_{\tau}$.
}
\label{non-tracing}
\end{figure}

 An important tool is the following lemma of Rezaei and Zhan \cite{MR3606738}. We define the function $P_y: [0,\infty) \to \R$ by
 \begin{align*}
P_y(x)= y^{8/\kappa-2+\kappa/8} x^{1-\kappa/8} \quad\text{if}\quad x\le y \quad\text{and}\quad P_y(x)=x^{8/\kappa-1}\quad\text{if}\quad x\ge y.
\end{align*}

\begin{lemma}[Theorem 1.1, \cite{MR3606738}] \label{lem:RZ}
Let $a_0, \ldots, a_L$ be distinct points in $\overline\h$ such that $a_0=0$. Let $y_j=\im(a_j) \ge 0$ and $l_j=\dist(a_j, \{a_m: 0\le m< j \}), 1\le j\le L.$ Let $r_1, \ldots, r_L>0$. Let $\gamma$ be an $\SLE_\kappa$ curve in $\h$ from $0$ to $\infty$. Then there is a constant $C_L<\infty$ depending only on $\kappa$ and $L$ such that
\begin{align*}
\p[\dist (\gamma, a_j)\le r_j, 1\le j \le L] \le C_L \prod_{j=1}^L\frac{P_{y_j}(r_j\wedge l_j)}{P_{y_j}(l_j)}.
\end{align*}
\end{lemma}

 Let $J_{\tau}$ be the Loewner hull of $\eta[0,\tau]$ and let $g_{\tau}$ be a conformal map from $\h\setminus J_{\tau}$ onto~$\h$ which fixes~$\infty$ (as we will see later, we will aim to estimate a particular ratio which does not depend on the choice of $g_\tau$).  We reorder the arcs $A_{i,1}, \ldots, A_{i, L}$ in a measurable way w.r.t.\ $\CF_{\tau}$ so that the diameters of $g_{\tau}(A_{i,j})$ are decreasing in $j$ for $j\in[1, L]$.  Let~$r_j$ denote the diameter of $g_{\tau}(A_{i,j})$.  Let $a_0=0$ and for $j\in[1,L]$, let~$a_j$ be the leftmost point of $\overline{g_{\tau}(A_{i,j})} \cap \R$.  Then for all $j\in[0,L]$, we have that $y_j=\im(a_j)=0$ and that~$P_{y_j}$ is equal to the function $x\mapsto x^{8/\kappa-1}$.  Let $l_j:=\dist(a_j, \{a_m: 0\le m< j \})$ for $1\le j\le L$.  We first prove the following estimate.
\begin{lemma}\label{lem:estimate}
For any $j\in[1,L]$,  we have
\begin{align}\label{eqn:diam_dist_ratio}
r_j / l_j =O \left(\delta^{(\xi-\lambda)/4}\right).
\end{align}
\end{lemma}

\begin{proof}
We emphasize that the ratio in the left side of~\eqref{eqn:diam_dist_ratio} does not depend on the renormalization of $g_{\tau}$ (as long as $g_{\tau}$ fixes $\infty$), hence we can choose any $g_{\tau}$ that is the most convenient for us.
In fact, we will choose a different $g_\tau$ for each $j$. Now fix $j\in[1,L]$.
Let $O_{i,j}$ denote the infinite connected component of  $\h \setminus \left( J_\tau\cup A_{i,j}\right)$. The two endpoints of $A_{i,j}$ and $\infty$ divide $\partial O_{i,j}$ into three parts. There is a unique point $a\in O_{i,j}$ such that the harmonic measure seen from $a$ in $O_{i,j}$ of the three boundary parts are all equal to $1/3$.
See Figure~\ref{non-tracing} for an illustration. 
We can now fix $g_{\tau}$ to be the conformal map that sends $a$ to $i$ and $\infty$ to $\infty$.
Then in the image upper half-plane, seen from $i$, the harmonic measure of $g_{\tau} \left(A_{i,j}\right)$ is $1/3$ and the harmonic measures of the parts of the real line to the left and right of $g_{\tau} \left(A_{i,j}\right)$ are both $1/3$.

Note that under our normalization there exist absolute constants $0 < c_0 < c_1 < \infty$ such that $c_0 \leq \diam(g_\tau(A_{i,j})) \leq c_1$.  We also note that the distance between $a$ and $w_j$ is at most $C\delta^\xi$ for some absolute constant $C>0$. This is because we can apply the Beurling estimate to the circles of radii $C\delta^\xi$ and $\delta^\xi$ around $w_j$ and get that if the distance from $a$ to $w_j$ is at least $C \delta^\xi$ for a large enough constant $C$, then a Brownian motion started from $a$ would have probability less than $1/3$ to stop in $A_{i,j}$.

Without loss of generality, we can assume that $A_{i,0}$ is attached to the left side of $J_\tau$. Then the harmonic measure seen from $a$ of the part of $\partial (\h \setminus J_\tau)$ which is to the right of $\eta(\tau)$ is  $O(\delta^{(\xi-\lambda)/2})$ as $a\in B(w_j, C\delta^\xi)$ and if we start a Brownian motion at $a$ and stop it upon hitting $\partial (\h \setminus J_\tau)$, then in order for it to stop on $\partial (\h \setminus J_\tau)$ to the right of $\eta(\tau)$, it has to travel distance at least $\delta^\lambda$ before exiting $\h \setminus J_\tau$.  By the Beurling estimate, this probability is $O(\delta^{(\xi-\lambda)/2})$.

Fix $m\in[1,j-1]$. The harmonic measure seen from $a$ of $A_{i,m}$ in $\h\setminus (J_{\tau} \cup A_{i,m})$ is $O(\delta^{(\xi-\lambda)/2})$ for similar reasons. Indeed, we know that $a\in B(w_j, C\delta^\xi)$ and if we start a Brownian motion at $a$ and stop it upon hitting $\partial (\h \setminus J_\tau) \cup A_{i,m}$, then in order for it to stop on $A_{i,m}$, it has to travel distance at least $\delta^\lambda$ before exiting the domain.  By the Beurling estimate, this probability is $O(\delta^{(\xi-\lambda)/2})$.
Since the diameters of $g_\tau(A_{i,j})$ are decreasing in $j$, we have that $\diam(g_\tau(A_{i,m})) \ge \diam (g_\tau(A_{i,j}))\ge c_0$. We thus see that the distance between $g_\tau(A_{i,j})$ and $g_\tau(A_{i,m})$ is at least a constant times $\delta^{-(\xi-\lambda)/4}$.
Since this is true for all $m\in[1,j-1]$, we have proved that under this normalization, $l_j$ is at least a constant times $\delta^{-(\xi-\lambda)/4}.$ Since $r_j=O(1)$, this implies the lemma.
\end{proof}

Combining Lemmas~\ref{lem:RZ} and~\ref{lem:estimate}, we get that conditionally on $\{\tau<\infty\}$ (e.g., the event that $\eta|_{[t_i,\infty)}$ visits $A_{i,0}$) and on $\CF_{\tau}$, the probability that $\eta|_{[\tau, \infty)}$ visits all the arcs $A_{i,j}$ for $j\in[1,L]$ is at most
\begin{align*}
C_L \prod_{j=1}^L (r_j/l_j)^{8/\kappa-1}= O\!\left(\delta^{L (8/\kappa-1) (\xi-\lambda)/4}\right).
\end{align*}
Since we further have that conditionally on $\{t_i<\infty\}$ and on $\CF_{t_i}$, the probability of $\{\tau<\infty\}$ is at most $1$, we have proved Lemma~\ref{lem:case1_goal}.

\subsubsection*{Case 2: $\CT$ contains a branch of depth at least $L+1$}
Let us fix some notation that is locally used in the present case.  Let this branch be $w_0, w_1,\ldots, w_{L}$ such that for all $j\in[1,L]$, $w_j$ is the child of $w_{j-1}$.  
In order not to deal with the possibility of $w_0$ being the artificial root vertex, we will only look at the branch $w_1, \ldots, w_L$. 
For $j\in[1,L]$, we denote the arcs $A_i(w_j)$ by $A_{i,j}$.
We aim to show the following lemma. See Figure~\ref{non-tracing2}.
\begin{lemma}\label{lem:case2_goal}
Conditionally on $\CF_{t_i}$, the probability that $\eta|_{[t_i, \infty)}$ visits the arc $A_{i,L}$  is $O(\delta^4)$.
 \end{lemma}

\begin{figure}[h!]
\centering
\includegraphics[scale=0.85]{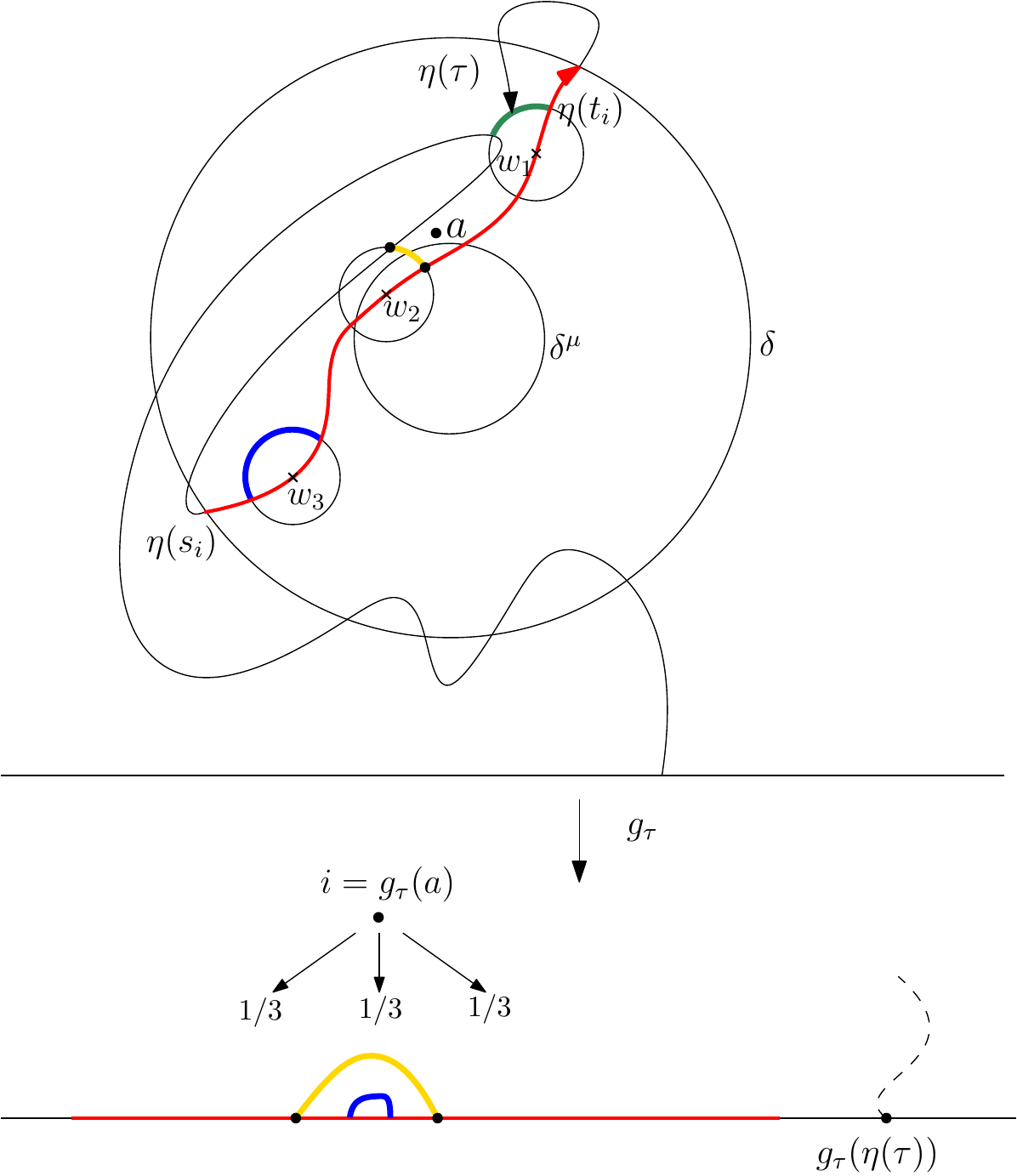}
\caption{Illustration of Case 2. On the top, we depict the excursion $e_i$ (in red) and the balls $B(w,\delta^\xi)$ where $w\in W_i$. We depict a branch of depth $3$: $w_1, w_2, w_3$.  The yellow arc centered at $w_j$ is chosen for the normalization of $g_{\tau}$.}
\label{non-tracing2}
\end{figure}

\begin{proof}
Note that Lemma~\ref{lem:case1_goal} is in fact valid for any $L\in\N$ (where $L$ need not satisfy~\eqref{eq:L}). If we apply Lemma~\ref{lem:case1_goal}  to $L=1$, then we have that for all $j\in[1,L]$, conditionally on $\CF_{t_i}$ and on the event that $\eta|_{[t_i,\infty)}$ visits $A_{i,j-1}$, the probability that $\eta|_{[t_i, \infty)}$ visits $A_{i,j}$ is $O\!\left(\delta^{ (8/\kappa-1) (\xi-\lambda)/4}\right)$.  In order for $\eta_{[t_i,\infty)}$ to hit $A_{i,L}$, it must successively hit $A_{i,j}$ for $j\in[1,L-1]$.  Therefore, conditionally on $\CF_{t_i}$, the probability that $\eta|_{[t_i,\infty)}$ intersects $A_{i,L}$ is
\[ O\!\left(\delta^{L (8/\kappa-1) (\xi-\lambda)/4}\right).\]
Due to~\eqref{eq:L}, this is also $O(\delta^4)$.
\end{proof}

Now we have treated the two cases and consequently completed the proof of Lemma~\ref{lem:S}. We can then deduce the following result on the set of all excursions of $\eta$ between $\partial B(z,\delta^\mu)$ and $\partial B(z,\delta)$.
\begin{lemma}\label{lem:all_excursions}
For any $z$ and $\delta>0$, the following event holds with probability $1-O(\delta^4)$: For each of the excursions $e_i$ that $\eta$ makes between $\partial B(z,\delta^\mu)$ and $\partial B(z,\delta)$, for any $(W_i, y^i)$ satisfying conditions (i)-(iii) of Lemma~\ref{def:W} with this value of $\delta$, one can find
 $S_i\subseteq W_i$ such that $|S_i|\ge R+1$ and  for all $w\in S_i$, $\eta|_{(t_i,\infty)}$ does not intersect $A\!\left(e_i,w,\delta^\xi, y^i(w)\right)$.
\end{lemma}
\begin{proof}
We can apply Lemma~\ref{lem:S} iteratively for each excursion $e_i$.
Note that conditionally on $\CF_{t_i}$, the probability that $\eta|_{[t_i,\infty)}$ returns and makes an $(i+1)$st excursion is $O(\delta^{c_0(\mu-1)})$, where $c_0$ is the constant in~\eqref{eq:pcrossing} (see \cite{MR2981906}).
In order for the event in Lemma~\ref{lem:all_excursions} to fail, one either fails to find the set $S_1$ for the first excursion which happens with probability $O(\delta^4)$, or $\eta|_{[t_1,\infty)}$ makes a second excursion but one fails to find the set $S_2$ which happens with probability $O( \delta^{c_0(\mu-1)}) O(\delta^4)$, etc.  The probability of failure is therefore at most
\begin{align*}
O(\delta^4)\sum_{k=0}^\infty \delta^{c_0(\mu-1)k}  = O(\delta^4).
\end{align*}
This concludes the proof.
\end{proof}

\begin{proof}[Proof of Lemma~\ref{lem:borel-cantelli}]
By the reversibility of SLE \cite{zhan2008reverse,ms2016imag3} for $\kappa \in (0,8)$, Lemma~\ref{lem:all_excursions} holds for the time-reversal of $\eta$.
Moreover, we can use the same $(W_i, y^i)$ for both the forward and reverse curves, since if $(W_i, y^i)$ is measurable w.r.t.\ $e_i$ as a set, then it is also measurable w.r.t.\ the time-reversal of $e_i$ as a set.
Thus with probability $1-O(\delta^4)$, the following holds:
In the forward direction, for each $W_i$, the number of  $w\in W_i$ such that $A\!\left(e_i,w,\delta^\xi, y^i(w)\right)\cap\eta(t_i,\infty)\not=\emptyset$ is at most $|W_i \setminus S_i|= R-1$. In the reverse direction,  for each $W_i$, the number of  $w\in W_i$ such that $A\!\left(e_i,w,\delta^\xi, y^i(w)\right)\cap\eta(0, s_i)\not=\emptyset$ is also at most $R-1$. This implies that there is at least one point $w\in W_i$ such that $A\!\left(e_i,w,\delta^\xi, y^i(w)\right)\cap\left(\eta(0, s_i)\cup \eta(t_i,\infty)\right)=\emptyset$.  This in fact means that $A\!\left(e_i,w,\delta^\xi, y^i(w)\right)\cap \eta=\emptyset$. The pair $(w, y^i(w))$ hence satisfies~\eqref{condition2} for the excursion $e_i$.  This completes the proof.
\end{proof}

\bibliographystyle{abbrv}
\bibliography{sle_welding_determined}

\end{document}